\documentclass{article}

\usepackage[a4paper,top=2.54cm,bottom=2.54cm,left=3.17cm,right=3.17cm,%
            includehead,includefoot]{geometry}

\usepackage{amsmath,amssymb,amsfonts,amsthm}
\usepackage{graphicx}
\usepackage{float}
\usepackage{pifont}
\usepackage{xcolor}
\usepackage[numbers,square,sort&compress]{natbib}
\usepackage{hyperref}
  \hypersetup{colorlinks,citecolor=blue,linkcolor=blue,breaklinks=true}
\usepackage{epstopdf}
\usepackage{yhmath} 
\usepackage{mathrsfs}

\allowdisplaybreaks
\numberwithin{equation}{section}
\newtheorem{theorem}{Theorem}[section]
\newtheorem{assumption}[theorem]{Assumption}
\newtheorem{lemma}[theorem]{Lemma}

\newtheorem{proposition}[theorem]{Proposition}
\newtheorem{definition}[theorem]{Definition}

\newtheorem{remark}[theorem]{Remark}

\newcommand{\bbm}{\begin{bmatrix}}
\newcommand{\ebm}{\end{bmatrix}}

\begin{document}

\title{\bf\Large Weak type estimates for Bochner--Riesz means on Hardy-type spaces associated with ball quasi-Banach function spaces
\footnotetext{\hspace{-0.35cm} 2010 {\it Mathematics Subject Classification}. Primary 42B30;
Secondary 42B25, 42B20.\endgraf
 \textbf{Keywords}\quad Ball quasi-Banach function space, Bochner--Riesz means, maximal Bochner--Riesz means, weak Hardy space, Lorentz space, Morrey space.}}

\author{
Jian Tan\footnote{Corresponding author} 
  \and
Linjing Zhang
}

\date{}

\maketitle

\noindent \textbf{Abstract}\quad Let $X\left(\mathbb{R}^{n}\right)$ be a ball quasi-Banach function space on $\mathbb{R}^{n}$, $WX\left(\mathbb{R}^{n}\right)$ be the weak ball quasi-Banach function space on $\mathbb{R}^{n}$, $H_{X}\left(\mathbb{R}^{n}\right)$ be the Hardy space associated with $X\left(\mathbb{R}^{n}\right)$ and $WH_{X}\left(\mathbb{R}^{n}\right)$ be the weak Hardy space associated with $X\left(\mathbb{R}^{n}\right)$. In this paper, we obtain the boundedness of the Bochner--Riesz means and the maximal Bochner--Riesz means from $H_{X}\left(\mathbb{R}^{n}\right)$ to $WH_{X}\left(\mathbb{R}^{n}\right)$ or $WX\left(\mathbb{R}^{n}\right)$, which includes the critical case. Moreover, we apply these results to several examples of ball quasi-Banach function spaces, namely, weighted Lebesgue spaces, Herz spaces, Lorentz spaces, variable Lebesgue spaces and Morrey spaces. This shows that all the results obtained in this article are of wide applications, and more applications of these results are predictable.

\section{Introduction and statement of main results }
In this paper, we will study the boundedness of Bochner--Riesz means and the maximal Bochner--Riesz means on weak Hardy-type spaces associated with Ball quasi-Banach function spaces.
The Bochner--Riesz operator of order $\delta\in\left(0,\infty\right)$ is defined for Schwartz function $f$ on $\mathbb{R}^{n}$ by setting, for any $x\in\mathbb{R}^{n}$,
\begin{align*}
	B_{1/\varepsilon}^{\delta}\left(f\right)\left(x\right)=\left(f\ast\phi_{\varepsilon}\right)\left(x\right),
\end{align*}
where, for any $x\in\mathbb{R}^{n}$ and $\varepsilon\in\left(0,\infty\right)$, $\phi\left(x\right):=\left[ \left(1-\left|\cdot\right|^{2} \right)_{+}^{\delta} \right]^{\wedge}\left(x\right)$ and $\phi_{\varepsilon}:=\varepsilon^{-n}\phi\left(\frac{x}{\varepsilon}\right)$. $B_{1/\varepsilon}^{\delta}\left(f\right)\left(x\right)$ is usually called the Bochner--Riesz means of $f$ with $\delta$ order on $\mathbb{R}^{n}$, and $\phi\left(x\right)$ is called the kernel of this convolution operator. Moreover, denote by
\begin{align*}
	B_{\ast}^{\delta}\left(f\right)\left(x\right)=\sup\limits_{\varepsilon>0}\left|B_{1/\varepsilon}^{\delta}\left(f\right)\left(x\right)\right|
\end{align*}
the maximal Bochner--Riesz means. By the definition of the radial maximal function, it is easy to see that $B_{\ast}^{\delta}\left(f\right)\left(x\right)=\phi_{+}^{\ast}\left(f\right)\left(x\right)$.
The Bochner--Riesz means were first introduced by Bochner \cite{Bochner} in connection with summation of multiple Fourier series and later played a vital role in harmonic analysis. The problem concerning the convergence of multiple Fourier series have led to the study of their $L^{p}$ boundedness. Next, we will present some historical results in regard to the boundedness of Bochner-Riesz means and the maximal Bochner--Riesz means. The following $H^{p}\left(\mathbb{R}^{n}\right)$ boundedness of the Bochner--Riesz means was proved by Sj{\"o}lin \cite{Sj} and Stein, Taibleson and Weiss \cite{SteinTaiblesonWeiss}.

\begin{theorem}
	\textup{Suppose that $0<p\le 1$ and $\delta>n/p-\left(n+1\right)/2$. Then there exists a constant $C>0$ independent of $f$ and $\varepsilon$ such that}
	\begin{align*}
		\left\|B_{1/\varepsilon}^{\delta}\left(f\right)\right\|_{H^{p}\left(\mathbb{R}^{n}\right)}\le C\left\|f\right\|_{H^{p}\left(\mathbb{R}^{n}\right)}.
	\end{align*}
\end{theorem}

Stein, Taibleson and Weiss \cite{SteinTaiblesonWeiss} also obtained the following weak type estimate for the maximal Bochner--Riesz means on $H^{p}\left(\mathbb{R}^{n}\right)$.

\begin{theorem}
	\textup{Suppose that $0<p<1$ and $\delta=n/p-\left(n+1\right)/2$. Then there exists a constant $C>0$ independent of $f$ such that}
	\begin{align*}
		\sup\limits_{\lambda>0}\lambda^{p}\left|\left\{x\in\mathbb{R}^{n}:B_{\ast}^{\delta}\left(f\right)\left(x\right)>\lambda\right\}\right|\le C\left\|f\right\|_{H^{p}\left(\mathbb{R}^{n}\right)}^{p}.
	\end{align*}
\end{theorem}

In 1995, Sato \cite{Sato} considered the following weighted weak type estimate of the maximal Bochner-Riesz means at critical index $\delta=n/p-\left(n+1\right)/2$.

\begin{theorem}
	\textup{Let $0<p<1$, $\omega\in A_{1}$ and $\delta=n/p-\left(n+1\right)/2$. Then, for all $f\in H_{\omega}^{p}\left(\mathbb{R}^{n}\right)$,}
	\begin{align*}
		\sup\limits_{\lambda>0}\lambda^{p}\omega\left(\left\{x\in\mathbb{R}^{n}:B_{\ast}^{\delta}\left(f\right)\left(x\right)>\lambda\right\}\right)\le C_{\left(\omega,n,p\right)}\left\|f\right\|_{H_{\omega}^{p}\left(\mathbb{R}^{n}\right)}^{p},
	\end{align*}
    \textup{where $\omega$ is a Muckenhoupt weight and $H_{\omega}^{p}\left(\mathbb{R}^{n}\right)$ is the weighted Hardy spaces.}
\end{theorem}

In 2006, Lee \cite{Lee} studied the boundedness of the maximal Bochner--Riesz means from $H_{\omega}^{p}\left(\mathbb{R}^{n}\right)$ to $L_{\omega}^{p}\left(\mathbb{R}^{n}\right)$ and the $H_{\omega}^{p}\left(\mathbb{R}^{n}\right)$ boundedness of the Bochner-Riesz means, respectively, as follows.

\begin{theorem}\label{th:Lee01}
	\textup{Let $\omega\in A_{1}$. Suppose that $0<p\le 1$ and $\delta>n/p-\left(n+1\right)/2$. Then there exists a constant $C>0$ independent of $f$ such that}
	\begin{align*}
		\left\|B_{\ast}^{\delta}\left(f\right)\right\|_{L_{\omega}^{p}\left(\mathbb{R}^{n}\right)}\le C\left\|f\right\|_{H_{\omega}^{p}\left(\mathbb{R}^{n}\right)}.
	\end{align*}
\end{theorem}

\begin{theorem}\label{th:Lee02}
	\textup{Let $\omega\in A_{1}$ with critical index $r_{\omega}$ for the reverse H{\"o}lder condition, $0<p\le 1$, $\delta>\max{\left\{n/p-\left(n+1\right)/2,\left[n/p\right]r_{\omega}/\left(r_{\omega}-1\right)-\left(n+1\right)/2\right\}}$. Then there exists a constant $C>0$ independent of $f$ and $\varepsilon$ such that}
	\begin{align*}
		\left\|B_{1/\varepsilon}^{\delta}\left(f\right)\right\|_{H_{\omega}^{p}\left(\mathbb{R}^{n}\right)}\le C\left\|f\right\|_{H_{\omega}^{p}\left(\mathbb{R}^{n}\right)}.
	\end{align*}
\end{theorem}

The above Theorems \ref{th:Lee01} and \ref{th:Lee02}  for the general weighted Hardy spaces $H_{\omega}^{p}\left(\mathbb{R}^{n}\right)$ were extended by Tan \cite{Tan}. Furthermore, Tan \cite{Tan} also obtained the boundnedness of $B_{1/\varepsilon}^{\delta}$ and $B_{\ast}^{\delta}$ on the variable Hardy spaces $H^{p(\cdot)}\left(\mathbb{R}^{n}\right)$.
Wang \cite{Wang} proved the boundedness of $B_{1/\varepsilon}^{\delta}$ and $B_{\ast}^{\delta}$ for the weighted weak Hardy spaces $WH_{\omega}^{p}\left(\mathbb{R}^{n}\right)$ as follows.

\begin{theorem}
	\textup{Let $0<p\le 1$, $\delta>n/p-\left(n+1\right)/2$ and $\omega\in A_{1}$. Then there exists a constant $C>0$ independent of $f$ such that}
	\begin{align*}
		\left\|B_{\ast}^{\delta}\left(f\right)\right\|_{WL_{\omega}^{p}\left(\mathbb{R}^{n}\right)}\le C\left\|f\right\|_{WH_{\omega}^{p}\left(\mathbb{R}^{n}\right)}.
	\end{align*}
\end{theorem}

\begin{theorem}
	\textup{Let $0<p\le 1$, $\delta>n/p-\left(n+1\right)/2$ and $\omega\in A_{1}$. Suppose that $\delta-\left(n-1\right)/2$ is not a positive integer, then there exists a constant $C>0$ independent of $f$ and $\varepsilon$ such that}
	\begin{align*}
		\left\|B_{1/\varepsilon}^{\delta}\left(f\right)\right\|_{WH_{\omega}^{p}\left(\mathbb{R}^{n}\right)}\le C\left\|f\right\|_{WH_{\omega}^{p}\left(\mathbb{R}^{n}\right)}.
	\end{align*}
\end{theorem}

In 2019, Wang et al. \cite{WangLiuWang} obtained the boundedness of $B_{1/\varepsilon}^{\delta}$ for the weak Musielak-Orlicz Hardy spaces $WH^{\varphi}\left(\mathbb{R}^{n}\right)$.  The boundedness of $B_{\ast}^{\delta}$ for weak Musielak-Orlicz Hardy spaces was established in \cite{WangQiuWang}. Moreover, Ho studied the boundedness of $B_{\ast}^{\delta}$ on the weighted Hardy spaces with variable exponents $H_{\omega}^{p\left(\cdot\right)}\left(\mathbb{R}^{n}\right)$, the Hardy local Morrey spaces with variable exponents $HL\mathcal{M}_{u}^{p\left(\cdot\right)}\left(\mathbb{R}^{n}\right)$, the Orlicz-slice Hardy spaces $\left(HE_{\Phi}^{q}\right)_{t}\left(\mathbb{R}^{n}\right)$ and the Hardy--Morrey spaces with variable exponents $\mathcal{M}\mathcal{H}_{p\left(\cdot\right),u}\left(\mathbb{R}^{n}\right)$ respectively in \cite{Ho2019}, \cite{Ho2021}, \cite{Ho2021-2} and \cite{Hotoappear}.
It is well known that the frame of quasi-Banach function spaces fails to include some function spaces such as the Morrey spaces. Therefore, together with the motivation of extending quasi-Banach function spaces, Sawano et al. \cite{SawanoHoYang} introduced ball quasi-Banach function spaces. Let $X(\mathbb{R}^{n})$ be a ball quasi-Banach function space. Moreover, Sawano et al. \cite{SawanoHoYang} also introduced the Hardy space $H_{X}\left(\mathbb{R}^{n}\right)$ (see Definition \ref{de:the Hardy type space} below) by the Peetre-type maximal function. Furthermore, Zhang et al. \cite{ZhangYangYuan} introduced the weak Hardy-type space $WH_{X}\left(\mathbb{R}^{n}\right)$ (see Definition \ref{de:the weak Hardy-type space} below) via the radial grand maximal function, and established some real-variable characterizations of $WH_{X}\left(\mathbb{R}^{n}\right)$.
We refer the reader to \cite{CWYZ,SYY,WYYZ} for more studies on (weak) Hardy spaces associated with $X(\mathbb{R}^{n})$.

In this article, we obtain the following weak type estimates for the Bochner--Riesz means and the maximal Bochner--Riesz means on the Hardy space $H_{X}\left(\mathbb{R}^{n}\right)$ associated with $X\left(\mathbb{R}^{n}\right)$.

\begin{theorem}\label{th:B-R Hx-WHx}
	\textup{Let $0<\theta<s \le 1$, $q\in\left(1,\infty\right)$ and $\delta>\frac{n-1}{2}$. Assume that $X$ is a ball quasi-Banach function space satisfying Assumptions \ref{as:Fefferman-Stein}, \ref{as:(X^{1/s})'} and \ref{as:M is bounded on (WX)^{1/r}}. Assume that $X^{1/s}$ is a ball Banach function space. Let $B_{1/\varepsilon}^{\delta}$ be a Bochner-Riesz means with $\delta$ order on $\mathbb{R}^{n}$. If $\theta\in\left[\frac{2n}{n+1+2\delta},1\right)$ and there exists a positive constant $C_{0}$ such that for any $\alpha\in\left(0,\infty\right)$ and any sequence $\left\{f_{j}\right\}_{j\in\mathbb{N}}\subset \mathscr{M}\left(\mathbb{R}^{n}\right)$,}
	\begin{align}\label{eq:weak-type Fefferman-Stein}
		\alpha\left\|\chi_{\left\{x\in\mathbb{R}^{n}:\left\{\sum_{j\in\mathbb{N}}^{}\left[M\left(f_{j}\right)\left(x\right) \right]^{\frac{n+1+2\delta}{2n}} \right\}^{\frac{2n}{n+1+2\delta}}>\alpha \right\}} \right\|_{X^{\frac{n+1+2\delta}{2n}}} \le C_{0}\left\|\left(\sum_{j\in\mathbb{N}}^{}\left|f_{j} \right|^{\frac{n+1+2\delta}{2n}} \right)^{\frac{2n}{n+1+2\delta}} \right\|_{X^{\frac{n+1+2\delta}{2n}}}, 
	\end{align}
	\textup{then $B_{1/\varepsilon}^{\delta}$ has a unique extension on $H_{X}\left(\mathbb{R}^{n}\right)$. Moreover, there exists a positive constant $C$ such that for any $f\in H_{X}\left(\mathbb{R}^{n}\right)$,}
	\begin{align}\label{eq:B-R Hx-WHx}
		\left\|B_{1/\varepsilon}^{\delta}\left(f\right)\right\|_{WH_{X}\left(\mathbb{R}^{n}\right)} \le C\left\|f\right\|_{H_{X}\left(\mathbb{R}^{n}\right)}.
	\end{align}
\end{theorem}

\begin{theorem}\label{th:B-R Hx-WX}
	\textup{Let $0<\theta<s \le 1$, $q\in\left(1,\infty\right)$ and $\delta>\frac{n-1}{2}$. Assume that $X$ is a ball quasi-Banach function space satisfying Assumptions \ref{as:Fefferman-Stein}, \ref{as:(X^{1/s})'} and \ref{as:M is bounded on (WX)^{1/r}}. Assume that $X^{1/s}$ is a ball Banach function space. Let $B_{1/\varepsilon}^{\delta}$ be a Bochner--Riesz means with $\delta$ order on $\mathbb{R}^{n}$. If $\theta\in\left[\frac{2n}{n+1+2\delta},1\right)$ and there exists a positive constant $C_{0}$ such that for any $\alpha\in\left(0,\infty\right)$ and any sequence $\left\{f_{j}\right\}_{j\in\mathbb{N}}\subset \mathscr{M}\left(\mathbb{R}^{n}\right)$, (\ref{eq:weak-type Fefferman-Stein}) holds true, then $B_{1/\varepsilon}^{\delta}$ has a unique extension on $H_{X}\left(\mathbb{R}^{n}\right)$. Moreover, there exists a positive constant $C$ such that for any $f\in H_{X}\left(\mathbb{R}^{n}\right)$,}
	\begin{align}\label{eq:B-R Hx-WX}
		\left\|B_{1/\varepsilon}^{\delta}\left(f\right)\right\|_{WX(\mathbb{R}^{n})} \le C\left\|f\right\|_{H_{X}\left(\mathbb{R}^{n}\right)}.
	\end{align}
\end{theorem}

\begin{theorem}\label{th:MB-R Hx-WX}
	\textup{Let $0<\theta<s \le 1$, $q\in\left(1,\infty\right)$ and $\delta>\frac{n-1}{2}$. Assume that $X$ is a ball quasi-Banach function space satisfying Assumptions \ref{as:Fefferman-Stein}, \ref{as:(X^{1/s})'} and \ref{as:M is bounded on (WX)^{1/r}}. Assume that $X^{1/s}$ is a ball Banach function space. Let $B_{\ast}^{\delta}$ be the maximal Bochner--Riesz means with $\delta$ order on $\mathbb{R}^{n}$. If $\theta\in\left[\frac{2n}{n+1+2\delta},1\right)$ and there exists a positive constant $C_{0}$ such that for any $\alpha\in\left(0,\infty\right)$ and any sequence $\left\{f_{j}\right\}_{j\in\mathbb{N}}\subset \mathscr{M}\left(\mathbb{R}^{n}\right)$, (\ref{eq:weak-type Fefferman-Stein}) holds true, then $B_{\ast}^{\delta}$ has a unique extension on $H_{X}\left(\mathbb{R}^{n}\right)$. Moreover, there exists a positive constant $C$ such that for any $f\in H_{X}\left(\mathbb{R}^{n}\right)$,}
	\begin{align}\label{eq:MB-R Hx-WX}
		\left\|B_{\ast}^{\delta}\left(f\right)\right\|_{WX(\mathbb{R}^{n})} \le C\left\|f\right\|_{H_{X}\left(\mathbb{R}^{n}\right)}.
	\end{align}
\end{theorem}

The remainder of this paper is organized as follows. In Section 2, we first recall some notions and known results concerning the (weak) ball quasi-Banach function spaces. Then we show some assumptions on the Hardy--Littlewood maximal operator and recall the notions of Hardy type spaces and weak Hardy type spaces respectively. Finally, we recall some lemmas concerning Bochner--Riesz means. In Section 3, we present the proofs of main results (Theorems \ref{th:B-R Hx-WHx}, \ref{th:B-R Hx-WX} and \ref{th:MB-R Hx-WX}) in this article. In Section 4, we apply all the results to several examples of ball quasi-Banach function spaces, namely, weighted Lebesgue spaces, Herz spaces, Lorentz spaces, variable Lebesgue spaces and Morrey spaces, respectively, in Subsections 4.1-4.5.

Finally, we make some conventions on notation. In this paper, we always denote by $C$ a positive constant which is independent of the main parameters, but it may vary from line to line. We also use $C_{\left(\alpha,\beta,...\right)}$ to denote a positive constant depending on the indicated parameters $\alpha,\beta,...$. The symbol $A\lesssim B$ means that $A\le CB$. If $A\lesssim B$ and $B\lesssim A$, then we write $A\sim B$. The symbol $\lfloor s \rfloor$ for $s\in\mathbb{R}^{n}$ denotes the largest integer not greater than $s$, and $\lceil s \rceil$ denotes the smallest integer not less than $s$. For any subset $E$ of $\mathbb{R}^{n}$, we denote by $E^{c}$ the set $\mathbb{R}^{n}\setminus E$ and by $\chi_{E}$ its characteristic function. Let $\mathbb{N}:= \left\{1,2,...\right\}$, $\mathbb{Z}_{+}:=\mathbb{N}\cup \left\{0\right\}$. For any $q\in \left[1,\infty\right]$, we denote by $q^{\prime}$ its conjugate exponent with $1/q+1/{q^{\prime}}=1$.

\section{Preliminaries}

In this section, we present some notions and known results concerning the ball quasi-Banach function spaces and the weak ball quasi-Banach function spaces. Then we show some assumptions on the Hardy--Littlewood maximal operator. In Definitions \ref{de:the Hardy type space} and \ref{de:the weak Hardy-type space}, we recall the notions of Hardy type spaces $H_{X}\left(\mathbb{R}^{n}\right)$ and weak Hardy type spaces $WH_{X}\left(\mathbb{R}^{n}\right)$ respectively. Finally, we recall some lemmas concerning Bochner-Riesz means.

\subsection{Weak ball quasi-Banach function spaces.}

Denote by $\mathscr{M}(\mathbb{R}^{n})$ the set of all measurable functions on $\mathbb{R}^{n}$. Before presenting the notion of weak ball quasi-Banach function spaces, we first recall the concepts of Banach function spaces and ball quasi-Banach function spaces. 
For the sake of simplicity, $Y(\mathbb{R}^{n})=:Y$, where $Y$ is any (quasi-)Banach function spaces or ball (quasi-)Banach function spaces on $\mathbb R^n$.

\begin{definition}\label{de:Banach function space}
	\textup{\citep[Chapter 1, Definition 1.1 and 1.3]{BennettSharpley} A Banach function space $Y\subset \mathscr{M}(\mathbb{R}^{n})$ is called a Banach function space if it satisfies}
	\begin{itemize}
		\item[$\left( {\rm \romannumeral1} \right)$]
		\textup{$\left\|f\right\|_{Y}=0$ if and only if $f=0$ almost everywhere;}
		\item[$\left( {\rm \romannumeral2} \right)$]
		\textup{$\left|g\right| \le \left|f\right|$ almost everywhere implies that $\left\|g\right\|_{Y} \le \left\|f\right\|_{Y}$;}
		\item[$\left( {\rm \romannumeral3} \right)$]
		\textup{$0 \le f_{m}\uparrow f$ almost everywhere implies that $\left\|f_{m}\right\|_{Y} \uparrow \left\|f\right\|_{Y}$;}
		\item[$\left( {\rm \romannumeral4} \right)$]
		\textup{$\chi_{E} \in Y$ for any measurable set $E \subset \mathbb{R}^{n}$ with finite measure;}
		\item[$\left( {\rm \romannumeral5} \right)$]
		\textup{for any measurable set $E \subset \mathbb{R}^{n}$ with finite measure, there exists a positive constant $C_{\left(E\right)}$, depending on $E$, such that, for all $f\in Y$,}
		\begin{align}\label{eq:definition Banach function space}
			\int_{E} \left|f\left(x\right)\right|\,dx \le C_{\left(E\right)} \left\|f\right\|_{Y},
		\end{align}
	\end{itemize}
	\textup{where $f$ and $g$ are measurable functions. In fact, the above conditions ensure that the norm $\left\| \cdot \right\|_{Y}$ is a Banach function norm.}
\end{definition}


For $x\in \mathbb{R}^{n}$ and $r\in \left(0,\infty\right)$, let $B\left(x,r\right) := \left\{y\in \mathbb{R}^{n}: \left|x-y\right|<r \right\}$, and
\begin{align}\label{eq:B(x,r)}
	\mathbb{B}:=\left\{B\left(x,r\right):x\in\mathbb{R}^{n}\ {\rm and}\ r\in\left(0,\infty\right)\right\}.
\end{align}
We now present the notion of ball quasi-Banach function spaces as follows \cite{SawanoHoYang}.

\begin{definition}\label{de:ball quasi-Banach function space}
	\textup{A quasi-Banach space $X\subset \mathscr{M}(\mathbb{R}^{n})$ is called a ball quasi-Banach function space if it satisfies}
	\begin{itemize}
		\item[$\left( {\rm \romannumeral1} \right)$]
		\textup{$\left\|f\right\|_{X}=0$ implies that $f=0$ almost everywhere;}
		\item[$\left( {\rm \romannumeral2} \right)$]
		\textup{$\left|g\right| \le \left|f\right|$ almost everywhere implies that $\left\|g\right\|_{X} \le \left\|f\right\|_{X}$;}
		\item[$\left( {\rm \romannumeral3} \right)$]
		\textup{$0 \le f_{m}\uparrow f$ almost everywhere implies that $\left\|f_{m}\right\|_{X} \uparrow \left\|f\right\|_{X}$;}
		\item[$\left( {\rm \romannumeral4} \right)$]
		\textup{$B\in \mathbb{B}$ implies that $\chi_{B}\in X$, where $\mathbb{B}$ is as in (\ref{eq:B(x,r)}).}
	\end{itemize}
\end{definition}

For any ball Banach function space $X$, the associate space (K{\"o}the dual) $X'$ is defined by setting
\begin{align}\label{eq:X'}
	X':=\left\{f\in\mathscr{M}\left(\mathbb{R}^{n}\right):\left\|f\right\|_{X'}:=\sup \left\{\left\|fg\right\|_{L^{1}\left(\mathbb{R}^{n}\right)}:g\in X,\left\|g\right\|_{X}=1 \right\}<\infty \right\},
\end{align}
where $\left\|\cdot \right\|_{X'}$ is called the associate norm of $\left\|\cdot \right\|_{X}$.

\begin{remark}\label{re:X'}
	\textup{By \citep[Proposition 2.3]{SawanoHoYang}, we know that, if $X$ is a ball Banach function space, then its associate space $X'$ is also a ball Banach function space.}
\end{remark}

Then we recall the notions of the convexity and the concavity of ball quasi-Banach function spaces.

\begin{definition}\label{de:convexity-concavity}
	\textup{\citep[Definition 1.d.3]{LindenstraussTzafriri} Let $X$ be a ball quasi-Banach function space and $p\in\left(0,\infty\right)$.}
	\begin{itemize}
		\item[$\left( {\rm \romannumeral1} \right)$]
		\textup{The $p$-convexification $X^{p}$ of $X$ is defined by setting}
		\begin{align*}
			X^{p}:=\left\{f\in\mathscr{M}\left(\mathbb{R}^{n}\right): \left|f\right|^{p}\in X \right\}
		\end{align*}
	    \textup{equipped with the quasi-norm $\left\|f\right\|_{X^{p}}:=\left\|\left|f\right|^{p} \right\|_{X}^{1/p}$.}
		\item[$\left( {\rm \romannumeral2} \right)$]
		\textup{The space $X$ is said to be $p$-concave if there exists a positive constant $C$ such that for any sequence $\left\{f_{j}\right\}_{j\in\mathbb{N}}$ of $X^{1/p}$,}
		\begin{align*}
			\sum_{j\in\mathbb{N}} \left\|f_{j}\right\|_{X^{1/p}} \le C\left\|\sum_{j\in\mathbb{N}} \left|f_{j}\right| \right\|_{X^{1/p}}.
		\end{align*}
	    \textup{Particularly, $X$ is said to be strictly $p$-concave when $C=1$.}
	\end{itemize}
\end{definition}

Next we present the notion of weak ball quasi-Banach function spaces as follows \citep[Definition 2.8 and Remark 2.9]{ZhangYangYuan}.

\begin{definition}\label{de:WX}
	\textup{Let $X$ be a ball quasi-Banach function space. The weak ball quasi-Banach function space $WX$ is defined to be the set of all measurable functions $f$ satisfying}
	\begin{align}\label{eq:WX}
		\left\|f\right\|_{WX}:=\sup\limits_{\alpha\in\left(0,\infty\right)} \left[\alpha\left\|\chi_{\left\{x\in\mathbb{R}^{n}:\left|f\left(x\right)\right|>\alpha \right\}}\right\|_{X} \right]<\infty.
	\end{align}
\end{definition}

\begin{remark}\label{re:X-subset-WX}
	\begin{itemize}
		\item[$\left( {\rm \romannumeral1} \right)$]
		\textup{Let $X$ be a ball quasi-Banach function space. For any $f\in X$ and $\alpha\in\left(0,\infty\right)$, we have $\chi_{\left\{x\in\mathbb{R}^{n}:\left|f\left(x\right)\right|>\alpha\right\}}\left(x\right)\le \left|f\left(x\right)\right|/\alpha$ for any $x\in\mathbb{R}^{n}$, which, together with Definition \ref{de:ball quasi-Banach function space} $\left({\rm \romannumeral2} \right)$, further implies that}
		\begin{align*}
			\sup\limits_{\alpha\in\left(0,\infty\right)} \left[\alpha\left\|\chi_{\left\{x\in\mathbb{R}^{n}:\left|f\left(x\right)\right|>\alpha \right\}}\right\|_{X} \right] \le \left\|f\right\|_{X}.
		\end{align*}
		\textup{This shows that $X\subset WX$.}
        \item[$\left( {\rm \romannumeral2} \right)$]
		\textup{Let $f,g\in WX$ with $\left|f\right|\le\left|g\right|$. By Definition \ref{de:ball quasi-Banach function space} $\left({\rm \romannumeral2} \right)$, we conclude that $\left\|f\right\|_{WX}\le\left\|g\right\|_{WX}$.}
	\end{itemize}
\end{remark}

\begin{lemma}\label{le:quasi-norm on WX}
	\textup{\citep[Lemma 2.10]{ZhangYangYuan} Let $X$ be a ball quasi-Banach function space. Then $\left\|\cdot\right\|_{WX}$ is a quasi-norm on $WX$, namely,}
	\begin{itemize}
		\item[$\left( {\rm \romannumeral1} \right)$]
		\textup{$\left\|f\right\|_{WX}=0$ if and only if $f=0$ almost everywhere.}
		\item[$\left( {\rm \romannumeral2} \right)$]
		\textup{For any $\lambda\in\mathbb{C}$ and $f\in WX$,}
		\begin{align*}
			\left\|\lambda f\right\|_{WX}=\left|\lambda\right|\left\|f\right\|_{WX}.
		\end{align*}
		\item[$\left( {\rm \romannumeral3} \right)$]
		\textup{For any $f,g\in WX$, there exists a positive constant $C$ such that}
		\begin{align*}
			\left\|f+g\right\|_{WX} \le C\left[\left\|f\right\|_{WX}+\left\|g\right\|_{WX} \right].
		\end{align*}
	    \textup{Moreover, if $p\in\left(0,\infty\right)$ and $X^{1/p}$ is a ball Banach function space, then}
	    \begin{align*}
	    	\left\|f+g\right\|_{WX}^{1/p} \le 2^{\max{\left\{1/p,1\right\}}} \left[\left\|f\right\|_{WX}^{1/p}+\left\|g\right\|_{WX}^{1/p} \right].
	    \end{align*}
	\end{itemize}
\end{lemma}

\begin{lemma}\label{le:X-WX}
	\textup{\citep[Lemma 2.13]{ZhangYangYuan} Let $X$ be a ball quasi-Banach function space. Then the space $WX$ is also a ball quasi-Banach function space.}
\end{lemma}

Now, we recall the notions of Muckenhoupt weights $A_{p}\left(\mathbb{R}^{n}\right)$ in \cite{Grafakos}.

\begin{definition}\label{de:Muckenhoupt weights}
	\textup{An $A_{p}\left(\mathbb{R}^{n}\right)$-weight $\omega$, with $p\in\left[1,\infty\right)$, is a locally integrable and non-negative function on $\mathbb{R}^{n}$ satisfying that, when $p\in\left(1,\infty\right)$,}
	\begin{align*}
		\sup\limits_{B\in\mathbb{B}} \left[\frac{1}{\left|B\right|} \int_{B} \omega\left(x\right)\,dx \right] \left[\frac{1}{\left|B\right|} \int_{B} \left\{\omega\left(x\right)\right\}^{\frac{1}{1-p}}\,dx \right]^{p-1} <\infty
	\end{align*}
    \textup{and, when $p=1$,}
    \begin{align*}
    	\sup\limits_{B\in\mathbb{B}} \frac{1}{\left|B\right|} \int_{B} \omega\left(x\right)\,dx \left[\left\|\omega^{-1}\right\|_{L^{\infty}\left(B\right)} \right]<\infty,
    \end{align*}
    \textup{where $\mathbb{B}$ is as in (\ref{eq:B(x,r)}). Define $A_{\infty}\left(\mathbb{R}^{n}\right):=\bigcup_{p\in\left[1,\infty\right)} A_{p}\left(\mathbb{R}^{n}\right)$.}
\end{definition}

The following lemma is a powerful consequence of the reverse H{\"o}lder property of $A_{p}$ weights, which is a characterization of all $A_{1}$ weights.

\begin{lemma}\label{le:characterization of all A_1 weights}
	\textup{\citep[Theorem 7.2.7]{Grafakos} Let $\omega$ be an $A_{1}$ weight. Then there exist $0<\varepsilon<1$, a non-negative function $k$ such that $k,k^{-1}\in L^{\infty}$, and a non-negative locally integrable function $f$ that satisfies $M\left(f\right)<\infty$ a.e. such that}
	\begin{align}\label{eq:characterization of all A_1 weights}
		\omega\left(x\right)=k\left(x\right)M\left(f\right)\left(x\right)^{\varepsilon}.
	\end{align}
    \textup{Conversely, given a non-negative function $k$ such that $k,k^{-1}\in L^{\infty}$ and given a non-negative locally integrable function $f$ that satisfies $M\left(f\right)<\infty$ a.e., define $\omega$ via (\ref{eq:characterization of all A_1 weights}). Then $\omega$ is an $A_{1}$ weight that satisfies}
    \begin{align*}
    	\left[\omega\right]_{A_{1}}\le \frac{C_{n}}{1-\varepsilon} \left\|k\right\|_{L^{\infty}} \left\|k^{-1}\right\|_{L^{\infty}},
    \end{align*}
    \textup{where $C_{n}$ is a universal dimensional constant.}
\end{lemma}

\begin{definition}\label{de:the weighted Lebesgue space}
	\textup{Let $p\in\left(0,\infty\right)$ and $\omega\in A_{\infty}\left(\mathbb{R}^{n}\right)$. The weighted Lebesgue space $L_{\omega}^{p}\left(\mathbb{R}^{n}\right)$ is defined to be the set of all measurable functions $f$ such that}
	\begin{align*}
		\left\|f\right\|_{L_{\omega}^{p}\left(\mathbb{R}^{n}\right)}:=\left[\int_{\mathbb{R}^{n}} \left|f\left(x\right)\right|^{p}\omega\left(x\right)\,dx \right]^{\frac{1}{p}}<\infty.
	\end{align*}
\end{definition}

The following lemma plays a vital role in the proof of Theorems \ref{th:B-R Hx-WHx}, \ref{th:B-R Hx-WX} and \ref{th:MB-R Hx-WX} below.

\begin{lemma}\label{le:X-subset-the weighted Lebesgue space}
	\textup{\citep[Lemma 2.17]{ZhangYangYuan} Let $X$ be a ball quasi-Banach function space. Assume that there exists an $s\in\left(0,\infty\right)$ such that $X^{1/s}$ is a ball Banach function space and $M$ is bounded on $\left(X^{1/s}\right)'$. Then there exists an $\epsilon\in\left(0,1\right)$ such that $X$ continuously embeds into $L_{\omega}^{s}\left(\mathbb{R}^{n}\right)$ with $\omega:=\left[M\left(\chi_{B\left({\vec 0}_{n},1\right)}\right)\right]^{\epsilon}$, namely, there exists a positive constant $C$ such that for any $f\in X$,}
	\begin{align*}
		\left\|f\right\|_{L_{\omega}^{s}\left(\mathbb{R}^{n}\right)} \le C\left\|f\right\|_{X}.
	\end{align*}
\end{lemma}


\subsection{Assumptions on the Hardy--Littlewood maximal operator.}

Denote by $L_{{\rm loc}}^{1}\left(\mathbb{R}^{n} \right)$ the set of all locally integrable functions on $\mathbb{R}^{n}$. Recall that the Hardy-Littlewood maximal operator $M$ is defined by setting, for all $f\in L_{{\rm loc}}^{1}\left(\mathbb{R}^{n}\right)$ and $x\in \mathbb{R}^{n}$,
\begin{align}\label{eq:the Hardy-Littlewood maximal operator}
	Mf\left(x\right) := \sup\limits_{r\in\left(0,\infty\right)} \frac{1}{\left|B\left(x,r\right)\right|} \int_{B\left(x,r\right)} \left|f\left(y\right)\right|\,dy.
\end{align}

For any $\theta\in \left(0,\infty\right)$, the powered Hardy--Littlewood maximal operator $M^{\left(\theta\right)}$ is defined by setting, for all $f\in L_{{\rm loc}}^{1}\left(\mathbb{R}^{n}\right)$ and $x\in \mathbb{R}^{n}$,
\begin{align}\label{eq:the powered Hardy--Littlewood maximal operator}
	M^{\left(\theta\right)}\left(f\right)\left(x\right) := \left\{M\left(\left|f\right|^{\theta}\right)\left(x\right) \right\}^{1/\theta}.
\end{align}

In order to prove several theorems in this paper, we need the following assumptions. 

\begin{assumption}\label{as:the boundedness of M on X^1/p}
	\textup{Let $X$ be a ball quasi-Banach function space and there exists a $p_{-}\in\left(0,\infty\right)$ such that for any given $p\in\left(0,p_{-}\right)$ and $s\in\left(1,\infty\right)$, there exists a positive constant $C$ such that for any $\left\{f_{j}\right\}_{j=1}^{\infty}\subset \mathscr{M}\left(\mathbb{R}^{n}\right)$,}
	\begin{align}\label{eq:the boundedness of M on X^1/p}
		\left\|\left\{\sum_{j\in\mathbb{N}} \left[M\left(f_{j}\right)\right]^{s} \right\}^{1/s}\right\|_{X^{1/p}} \le C\left\|\left\{\sum_{j\in\mathbb{N}}\left|f_{j}\right|^{s} \right\}^{1/s}\right\|_{X^{1/p}}.
	\end{align}
\end{assumption}

\begin{remark}\label{re:the boundedness of M on X^1/p}
	\textup{If $X:=L^{\widetilde{p}}\left(\mathbb{R}^{n}\right)$ with any given $\widetilde{p}\in\left(0,\infty\right)$, then $p_{-}=\widetilde{p}$ and (\ref{eq:the boundedness of M on X^1/p}) is just the well-known Fefferman--Stein vector-valued maximal inequality, which was originally established by Fefferman and Stein \citep[Theorem 1(a)]{FeffermanStein1971}.}
\end{remark}

\begin{assumption}\label{as:Fefferman-Stein}
	\textup{Let $X$ be ball quasi-Banach function space.
		For some $\theta,s\in\left(0,1\right]$ and $\theta<s$, there exists a positive constant $C$ such that, for any  $\left\{f_{j}\right\}_{j=1}^{\infty} \subset \mathscr{M}\left(\mathbb{R}^{n}\right)$,}
	\begin{align}\label{eq:Fefferman-Stein}
		\left\|\left\{\sum_{j=1}^{\infty}\left[M^{(\theta)}\left(f_{j}\right)\right]^{s}\right\}^{1 / s}\right\|_{X} \le C\left\|\left\{\sum_{j=1}^{\infty}\left|f_{j}\right|^{s}\right\}^{1 / s}\right\|_{X}.
	\end{align}
\end{assumption}

\begin{remark}\label{re:Fefferman-Stein}
	\textup{The inequality (\ref{eq:Fefferman-Stein}) is called the Fefferman--Stein vector-valued maximal inequality, and its version with $X:=L^{p}\left(\mathbb{R}^{n}\right)$ for $p\in \left(1,\infty\right)$, $\theta=1$ and $s\in\left(1,\infty\right]$ was originally established by Fefferman and Stein \citep[Theorem 1]{FeffermanStein1971}. Observe that, by \citep[Theorem 1]{FeffermanStein1971}, we know that (\ref{eq:Fefferman-Stein}) also holds true when $\theta$, $s\in \left(0,1\right]$, $\theta<s$, $X:=L^{p}\left(\mathbb{R}^{n}\right)$ and $p\in\left(\theta,\infty\right)$.}
\end{remark}

\begin{assumption}\label{as:(X^{1/s})'}
	\textup{Let $X$ be a ball quasi-Banach function space satisfying (\ref{eq:Fefferman-Stein}) for some $\theta,s\in\left(0,1\right]$. Let $d\ge \left\lfloor n(1/\theta-1)\right\rfloor$ be a fixed integer and $q\in\left(1,\infty\right]$. Assume that for any $f\in\mathscr{M}\left(\mathbb{R}^{n}\right)$,}
	\begin{align}\label{eq:(X^{1/s})'}
		\left\|M^{\left((q/s)'\right)}(f)\right\|_{\left(X^{1/s}\right)'} \le C\|f\|_{\left(X^{1/s}\right)'},
	\end{align}
    \textup{where the implicit positive constant is independent of $f$.}
\end{assumption}

\begin{assumption}\label{as:M is bounded on (WX)^{1/r}}
	\textup{Let $X$ be a ball quasi-Banach function space. Assume that there exists an $r\in\left(0,\infty\right)$ such that $M$ in (\ref{eq:the Hardy-Littlewood maximal operator}) is bounded on $\left(WX\right)^{1/r}$.}
\end{assumption}

\subsection{Weak Hardy type spaces.}

For any given ball quasi-Banach function space $X$, the Hardy type space $H_{X}\left(\mathbb{R}^{n}\right)$ \cite{SawanoHoYang} was introduced via the Peetre-type maximal function \cite{FeffermanStein1972}. The weak Hardy type space $WH_{X}\left(\mathbb{R}^{n}\right)$ \cite{ZhangYangYuan} was introduced by the radial grand maximal function $M_{N}^{0}\left(f\right)$ of $f$. We first recall the concept of the former.

\begin{definition}\label{de:the Hardy type space}
	\textup{Let $X$ be a ball quasi-Banach function space. Let $\Phi\in \mathcal{S}\left(\mathbb{R}^{n}\right)$ satisfy $\int_{\mathbb{R}^{n}}\Phi\left(x\right)\,dx \ne 0$ and $b\in \left(0,\infty\right)$ sufficiently large. Then the Hardy space $H_{X}\left(\mathbb{R}^{n}\right)$ associated with $X$ is defined as}
	\begin{align*}
		H_{X}\left(\mathbb{R}^{n}\right) := \left\{f\in \mathcal{S}^{\prime}\left(\mathbb{R}^{n}\right):\left\|f\right\|_{H_{X}\left(\mathbb{R}^{n}\right)} := \left\|M_{b}^{\ast\ast}\left(f,\Phi\right)\right\|_{X} < \infty \right\},
	\end{align*}
	\textup{where the maximal function $M_{b}^{\ast\ast}\left(f,\Phi\right)$ of Peetre type is defined by setting, for all $x\in \mathbb{R}^{n}$,}
	\begin{align}\label{eq:the Peetre-type maximal function}
		M_{b}^{\ast\ast}\left(f,\Phi\right)\left(x\right) := \sup\limits_{\left(y,t\right)\in \mathbb{R}_{+}^{n+1}} \frac{\left|\left(\Phi_{t}\ast f\right) \left(x-y\right) \right|}{\left(1+t^{-1}\left|y\right| \right)^{b}},
	\end{align}
	\textup{where $\mathbb{R}_{+}^{n+1} := \mathbb{R}^{n} \times \left(0,\infty\right)$.}
\end{definition}

\begin{definition}\cite[Definition 3.2]{WangYangYang} 
\textup{Let $X$ be a ball quasi-Banach function space. A function $f\in X$ is said to have an absolutely continuous quasi-norm in $X$ if
$\|f\chi_{E_j}\|_X\downarrow 0$
whenever $\{E_j\}_{j=1}^{\infty}$ is a sequence of measurable sets that satisfy
$E_{j+1}\subset E_j$ for any $j\in\mathbb N$ and $\cap_{j=1}^\infty E_j=\emptyset$.
Moreover, $X$ is said to have an absolutely
continuous quas-norm if, for any $f\in X$, $f$ has an absolutely continuous
quasi-norm in $X$.}
\end{definition}

Moreover, the atomic decomposition theory is very useful
when we consider the boundedness of operators on Hardy spaces.
For instance, see \cite{Tan19,Tan21,YZ}.
Now we recall the atomic characterization for the Hardy space $H_{X}\left(\mathbb{R}^{n}\right)$ established in \cite{SawanoHoYang}.

\begin{lemma}\label{le:H_X-H_atom equivalent quasi-norm}
	\textup{\citep[Lemma 6.2]{ZhangYangYuan} Let $\theta,s\in\left(0,1\right]$, $q\in\left(1,\infty\right]$ and $d\ge\left\lfloor n\left(1/\theta-1\right)\right\rfloor$ be a fixed integer. Assume that $X$ is a ball quasi-Banach function space satisfying (\ref{eq:Fefferman-Stein}), (\ref{eq:(X^{1/s})'}) and that $X^{1/s}$ is a ball Banach function space. Then $H_{X}\left(\mathbb{R}^{n}\right)=H_{atom}^{X,q,d}\left(\mathbb{R}^{n}\right)$ with equivalent quasi-norm.}
\end{lemma}

\begin{lemma}\label{le:the atomic characterization for the Hardy type space}
	\textup{\citep[Theorem 3.7]{SawanoHoYang} Let $X$ be a ball quasi-Banach function space satisfying Assumption \ref{as:Fefferman-Stein} with $0<\theta<s\le 1$. Assume further that $X$ has an absolutely continuous quasi-norm,}
	\begin{align}\label{eq:d_{X}}
		d_{X}:=\lceil n(1/\theta-1)\rceil
	\end{align}
	\textup{and $d\in \left[d_{X},\infty \right) \cap\mathbb{Z}_{+}$ is a fixed integer. Then, for any $f\in H_{X}\left(\mathbb{R}^{n}\right)$, there exist a sequence $\left\{a_{j}\right\}_{j=1}^{\infty}$ of $\left(X,\infty,d\right)$-atoms supported, respectively, in a sequence $\left\{Q_{j}\right\}_{j=1}^{\infty}$ of cubes, a sequence $\left\{\lambda_{j}\right\}_{j=1}^{\infty}$ of non-negative numbers and a positive constant $C_{(s)}$, independent of $f$ but depending on $s$, such that}
	\begin{align}\label{eq:atomic decomposition}
		f=\sum_{j=1}^{\infty} \lambda_{j} a_{j}\quad {\rm in} \quad \mathcal{S}^{\prime}\left(\mathbb{R}^{n}\right) 
	\end{align}
	\textup{and}
	\begin{align*}
		\left\|\left\{\sum_{j=1}^{\infty}\left(\frac{\lambda_{j}}{\left\|\chi_{Q_{j}}\right\|_{X}}\right)^{s} \chi_{Q_{j}}\right\}^{1/s}\right\|_{X} \le C_{(s)}\|f\|_{H_{X}\left(\mathbb{R}^{n}\right)}.
	\end{align*}
\end{lemma}

\begin{lemma}\label{le:atom2}
	\textup{\citep[Lemma 2.21]{WangYangYang} Assume that $X$ is a ball quasi-Banach function space satisfying Assumption \ref{as:Fefferman-Stein} with some $s\in\left(0,1\right]$ and Assumption \ref{as:(X^{1/s})'} for some $q\in\left(1,\infty \right]$ and the same $s\in\left(0,1\right]$ as in (\ref{eq:Fefferman-Stein}). Let $\left\{a_{j}\right\}_{j=1}^{\infty} \subset L^{q}\left(\mathbb{R}^{n}\right)$ be supported, respectively, in cubes $\left\{Q_{j}\right\}_{j=1}^{\infty}$, and a sequence $\left\{\lambda_{j}\right\}_{j=1}^{\infty} \subset[0, \infty)$ such that, for any $j\in \mathbb{N}$,}
	\begin{align*}
		\left\|a_{j}\right\|_{L^{q}\left(\mathbb{R}^{n}\right)} \le \frac{\left|Q_{j}\right|^{1 / q}}{\left\|\chi_{Q_{j}}\right\|_{X}}
	\end{align*}
	\textup{and}
	\begin{align*}
		\left\|\left\{\sum_{j=1}^{\infty}\left(\frac{\lambda_{j}}{\left\|\chi_{Q_{j}}\right\|_{X}}\right)^{s} \chi_{Q_{j}}\right\}^{1/s}\right\|_{X}<\infty. 
	\end{align*}
	\textup{Then $f=\sum_{j=1}^{\infty} \lambda_{j} a_{j}$ converges in $S^{\prime}\left(\mathbb{R}^{n}\right)$ and there exists a positive constant $C$, independent of $f$, such that}
	\begin{align*}
		\|f\|_{X} \le C\left\|\left\{\sum_{j=1}^{\infty}\left(\frac{\lambda_{j}}{\left\|\chi_{Q_{j}}\right\|_{X}}\right)^{s} \chi_{Q_{j}}\right\}^{1/s}\right\|_{X}.
	\end{align*}
\end{lemma}

\begin{remark}\label{re:atom2}
	\textup{We point out that Lemma \ref{le:the atomic characterization for the Hardy type space} need the additional assumption that there exists $s\in \left(0,1\right]$ such that $X^{1/s}$ is a ball Banach function space. Indeed, this assumption ensures that $\left(X^{1/s}\right)^{\prime}$ is also a ball Banach function space, which implies that, for any $f\in \left(X^{1/s}\right)^{\prime}$, $f\in L_{{\rm loc}}^{1}\left(\mathbb{R}^{n}\right)$ and hence the Hardy-Littlewood maximal operator can be defined on $\left(X^{1/s}\right)^{\prime}$ in \cite{WangYangYang}.}
\end{remark}

\begin{lemma}\label{le:equivalent-X}
	\textup{\citep[Theorem 3.1]{SawanoHoYang} Let $X$ be a ball quasi-Banach function space, $r\in\left(0,\infty\right)$, $b\in\left(n/r,\infty\right)$ and $M$ be bounded on $X^{1/r}$. Assume that $\Phi\in \mathcal{S}\left(\mathbb{R}^{n}\right)$ satisfies $\int_{\mathbb{R}^{n}} \Phi(x)\,dx\ne 0$. Then, for any $f\in \mathcal{S}^{\prime}\left(\mathbb{R}^{n}\right)$,}
	\begin{align*}
		\left\|M_{b}^{\ast\ast}\left(f,\Phi\right)\right\|_{X} \sim \left\|M\left(f,\Phi\right)\right\|_{X},
	\end{align*}
	\textup{where $M_{b}^{\ast\ast}\left(f,\Phi\right)$ is the Peetre-type maximal function, $M\left(f,\Phi\right) :=\sup\limits_{t\in\left(0,\infty\right)} \left|\Phi_{t} \ast f\right|$ and the positive equivalence constants are independent of $f$.}
\end{lemma}

\begin{lemma}\label{le:density}
	\textup{\citep[Corollary 3.11]{SawanoHoYang} Let $s\in \left(0,1\right]$, $q\in \left(1,\infty \right]$ . Assume that $X$ is a strictly $s$-convex ball quasi-Banach function space satisfying Assumption \ref{as:Fefferman-Stein} for some $\theta\in \left(0,1\right]$. Assume further that $X$ has an absolutely continuous quasi-norm. Then }
	\begin{itemize}
		\item[$\left( {\rm \romannumeral1} \right)$]
		\textup{$H_{X}\left(\mathbb{R}^{n}\right) \cap L^{\infty}\left(\mathbb{R}^{n}\right)$ is dense in $H_{X}\left(\mathbb{R}^{n}\right)$.}
		\item[$\left( {\rm \romannumeral2} \right)$]
		\textup{The convergence of (\ref{eq:atomic decomposition}) holds true in $H_{X}\left(\mathbb{R}^{n}\right)$.}
	\end{itemize}
\end{lemma}

\begin{remark}\label{re:density}
	\textup{If $X$ is as in Lemma \ref{le:density}, and let $f\in H_{X}\left(\mathbb{R}^{n}\right)$, $\left\{\lambda_{j}\right\}_{j=1}^{\infty}$ and $\left\{a_{j}\right\}_{j=1}^{\infty}$ be as in Lemma \ref{le:the atomic characterization for the Hardy type space}. Then, $H_{X}\left(\mathbb{R}^{n}\right) \cap L^{2}\left(\mathbb{R}^{n}\right)$ is dense in $H_{X}\left(\mathbb{R}^{n}\right)$.}
\end{remark}

For any $N\in\mathbb{N}$, let
\begin{align*}
	\mathcal{F}_{N}\left(\mathbb{R}^{n}\right):=\left\{\varphi\in\mathcal{S}\left(\mathbb{R}^{n}\right):\sum_{\beta\in\mathbb{Z}_{+}^{n},\left|\beta\right|\le N} \sup\limits_{x\in\mathbb{R}^{n}} \left[\left(1+\left|x\right|\right)^{N+n} \left|\partial_{x}^{\beta}\varphi\left(x\right)\right| \right] \le 1 \right\};
\end{align*}
here and thereafter, for any $\beta:=\left(\beta_{1},...,\beta_{n}\right)\in\mathbb{Z}_{+}^{n}$ and $x\in\mathbb{R}^{n}$, $\left|\beta\right|:=\beta_{1}+\ldots+\beta_{n}$ and
\begin{align*}
	\partial_{x}^{\beta}:=\left(\frac{\partial}{\partial x_{1}}\right)^{\beta_{1}}\ldots \left(\frac{\partial}{\partial x_{n}}\right)^{\beta_{n}}.
\end{align*}

\begin{definition}\label{de:the radial grand maximal function}
	\textup{For any given $f\in\mathcal{S}'\left(\mathbb{R}^{n}\right)$, the radial grand maximal function $M_{N}^{0}\left(f\right)$ of $f$ is defined by setting, for any $x\in\mathbb{R}^{n}$,}
	\begin{align}\label{eq:the radial grand maximal function}
		M_{N}^{0}\left(f\right)\left(x\right):=\sup \left\{\left|f\ast\varphi_{t}\left(x\right)\right|:t\in\left(0,\infty\right)\ and\ \varphi\in\mathcal{F}_{N}\left(\mathbb{R}^{n}\right) \right\},
	\end{align}
    \textup{where for any $t\in\left(0,\infty\right)$ and $\xi\in\mathbb{R}^{n}$, $\varphi_{t}\left(\xi\right):=t^{-n}\varphi\left(\xi/t\right)$.}
\end{definition}

\begin{definition}\label{de:the weak Hardy-type space}
	\textup{\citep[Definition 2.21]{ZhangYangYuan} Let $X$ be a ball quasi-Banach function space. Then the weak Hardy-type space $WH_{X}\left(\mathbb{R}^{n}\right)$ associated with $X$ is defined by setting}
	\begin{align*}
		WH_{X}\left(\mathbb{R}^{n}\right):=\left\{f\in\mathcal{S}'\left(\mathbb{R}^{n}\right):\left\|f\right\|_{WH_{X}\left(\mathbb{R}^{n}\right)}:=\left\|M_{N}^{0}\left(f\right)\right\|_{WX}<\infty \right\},
	\end{align*}
    \textup{where $M_{N}^{0}\left(f\right)$ is as in (\ref{eq:the radial grand maximal function}) with $N\in\mathbb{N}$ sufficiently large.}
\end{definition}

\begin{remark}\label{re:the weak Hardy-type space}
	\textup{When $X:=L^{p}\left(\mathbb{R}^{n}\right)$ with $p\in\left(0,1\right]$, the Hardy-type space $WH_{X}\left(\mathbb{R}^{n}\right)$ coincides with the classical weak Hardy space $WH^{p}\left(\mathbb{R}^{n}\right)$ \cite{Liu}.}
\end{remark}

\begin{lemma}\label{le:equivalent-WX}
	\textup{\citep[Theorem 3.2]{ZhangYangYuan} Let $X$ be a ball quasi-Banach function space, $r\in\left(0,\infty\right)$ and $b\in\left(n/r,\infty\right)$. Assume that the Hardy-Littlewood maximal operator $M$ in (\ref{eq:the Hardy--Littlewood maximal operator}) is bounded on $\left(WX\right)^{1/r}$. Let $\Phi\in \mathcal{S}\left(\mathbb{R}^{n}\right)$ satisfy $\int_{\mathbb{R}^{n}} \Phi(x)\,dx\ne 0$. Then, for any $f\in \mathcal{S}'\left(\mathbb{R}^{n}\right)$,}
	\begin{align*}
		\left\|M_{N}^{0}\left(f\right)\right\|_{WX} \sim \left\|M\left(f,\Phi\right)\right\|_{WX},
	\end{align*}
	\textup{where $M_{N}^{0}\left(f\right)$ is the Peetre-type maximal function as in (\ref{eq:the radial grand maximal function}), $M\left(f,\Phi\right) :=\sup\limits_{t\in\left(0,\infty\right)} \left|\Phi_{t} \ast f\right|$ and the positive equivalence constants are independent of $f$.}
\end{lemma}

\subsection{Bochner--Riesz means.}

The following properties of Bochner--Riesz means play a key role in the proof of the main results.

\begin{lemma}\label{le:phi(B-R kernel)} 
	\textup{\cite{Sato} If $\delta=\frac{n}{p}-\frac{n+1}{2}$ and $0<p<1$, then $\phi$ satisfies the inequality}
	\begin{align}\label{eq:B-R-phi}
		\sup\limits_{x\in \mathbb{R}^{n}} \left(1+\left|x\right|\right)^{n/p} \left|D^{\alpha}\phi\left(x\right)\right| \le C_{\left(\alpha,n,p\right)}
	\end{align}
	\textup{for all multi-indices $\alpha$.}
\end{lemma}

\begin{lemma}\label{le:phi^t}
	\textup{\cite{ZhangTan} Let $0<p_{0}<1$, $\delta=\frac{n}{p_{0}}-\frac{n+1}{2}$, and $\phi$ be a distribution that satisfies}
	\begin{itemize}
		\item[$\left( {\rm \romannumeral1} \right)$] $\left\|B_{1/\varepsilon}^{\delta}(f)\right\|_{L^{2}\left(\mathbb{R}^{n}\right)} \lesssim \left\|f\right\|_{L^{2}\left(\mathbb{R}^{n}\right)}$;
		\item[$\left( {\rm \romannumeral2} \right)$] $\left|\phi(x)\right| \lesssim \left(1+|x|\right)^{-n/{p_{0}}}$;
		\item[$\left( {\rm \romannumeral3} \right)$]
		\textup{Let $\operatorname{supp}{\left(\phi\right)} \subset \left\{\left|x\right|<1\right\}$, $N=\left[n\left(1/{p_{0}}-1\right)\right]$, and $P_{\phi}\left(\cdot\frac{x}{\varepsilon}\right)$ denote the Taylor polynomial of $N$-th order of $\phi\left(t\right)$ at $t=\frac{x}{\varepsilon}$. For $\left|x\right|\ge 4\left|y\right|$ and $\varepsilon>\frac{3}{4}\left|x\right|$,}
		\begin{align*}
			\varepsilon^{-n} \left|\phi\left(\frac{x-y}{\varepsilon}\right)-P_{\phi}\left(-\frac{y}{\varepsilon}\right) \right| \lesssim \frac{\left|y\right|^{N+1}}{\left|x\right|^{N+n+1}},
		\end{align*}
	\end{itemize}
	\textup{and define $\phi^{(t)}$ by $\phi^{(t)}=\phi \ast \Phi_{t}$. Then $\phi^{(t)}$ satisfies the same hypotheses, with bounds that may be chosen independent of $t$. In particular}
	\begin{align*}
		\left|\phi^{(t)}\left(x\right)\right| \lesssim \left(1+\left|x\right|\right)^{-n/{p_{0}}}
	\end{align*}
	\textup{and}
	\begin{align*}
		\varepsilon^{-n} \left|\phi^{\left(t\right)}\left(\frac{x-y}{\varepsilon}\right)-P_{\phi^{\left(t\right)}}\left(-\frac{y}{\varepsilon}\right) \right| \lesssim \frac{\left|y\right|^{N+1}}{\left|x\right|^{N+n+1}},
	\end{align*}
	\textup{whenever $|x| \ge 4|y|$ and $\varepsilon>\frac{3}{4}\left|x\right|$.}	
\end{lemma} 

\begin{lemma}\label{le:B-R H_w^s-WH_w^s}
	\textup{Let $s\in\left(0,\infty\right)$, $\omega\in A_{\infty}\left(\mathbb{R}^{n}\right)$ and $\delta>\frac{n-1}{2}$. Then there exists a positive constant $C$ such that, for any $f\in H_{\omega}^{s}\left(\mathbb{R}^{n}\right)$,}
	\begin{align}\label{eq:B-R H_w^s-WH_w^s}
		\left\|B_{1/\varepsilon}^{\delta}\left(f\right) \right\|_{WH_{\omega}^{s}\left(\mathbb{R}^{n}\right)} \le C\left\|f\right\|_{H_{\omega}^{s}\left(\mathbb{R}^{n}\right)},
	\end{align}
	\textup{where $H_{\omega}^{s}\left(\mathbb{R}^{n}\right)$ denotes the weighted Hardy space as in Definition \ref{de:the Hardy type space} with $X$ replaced by $L_{\omega}^{s}\left(\mathbb{R}^{n}\right)$, and $WH_{\omega}^{s}\left(\mathbb{R}^{n}\right)$ denotes the weighted weak Hardy space as in Definition \ref{de:the weak Hardy-type space} with $X$ replaced by $L_{\omega}^{s}\left(\mathbb{R}^{n}\right)$.}
\end{lemma}

\begin{proof}
	By \citep[Corollary 3.5({\rm \romannumeral2})]{ZhangTan}, we know that $\left\|B_{1/\varepsilon}^{\delta}\left(f\right) \right\|_{H_{\omega}^{s}\left(\mathbb{R}^{n}\right)} \le C\left\|f\right\|_{H_{\omega}^{s}\left(\mathbb{R}^{n}\right)}$ with $s\in\left(0,\infty\right)$. Therefore, to prove (\ref{eq:B-R H_w^s-WH_w^s}), it is only necessary to prove $\left\|B_{1/\varepsilon}^{\delta}\left(f\right) \right\|_{WH_{\omega}^{s}\left(\mathbb{R}^{n}\right)} \le C\left\|B_{1/\varepsilon}^{\delta}\left(f\right) \right\|_{H_{\omega}^{s}\left(\mathbb{R}^{n}\right)}$, which can be derived from $H_{\omega}^{s}\left(\mathbb{R}^{n}\right) \subset WH_{\omega}^{s}\left(\mathbb{R}^{n}\right)$. Then we will prove $H_{\omega}^{s}\left(\mathbb{R}^{n}\right) \subset WH_{\omega}^{s}\left(\mathbb{R}^{n}\right)$ as follows.
	
	Let $X$ be a ball quasi-Banach function space. We know that $L_{\omega}^{s}\left(\mathbb{R}^{n}\right)$ is a ball quasi-Banach function space. In order to take advantage of the properties of ball quasi-Banach function space in the following proof, we can directly prove that $H_{X}\subset WH_{X}$.
	
	From Definitions \ref{de:the Hardy type space}, \ref{de:the weak Hardy-type space} and Lemmas \ref{le:equivalent-X}, \ref{le:equivalent-WX}, we notice that
	\begin{align*}
		\left\|f\right\|_{H_{X}\left(\mathbb{R}^{n}\right)}=\left\|M_{b}^{\ast\ast}\left(f,\Phi\right)\right\|_{X}\sim \left\|M\left(f,\Phi\right)\right\|_{X}
	\end{align*}
	and
	\begin{align*}
		\left\|f\right\|_{WH_{X}\left(\mathbb{R}^{n}\right)}=\left\|M_{N}^{0}\left(f\right)\right\|_{WX}\sim \left\|M\left(f,\Phi\right)\right\|_{WX},
	\end{align*}
	which combined with $X\subset WX$ as in Remark \ref{re:X-subset-WX}, we conclude that
	\begin{align*}
		\left\|M\left(f,\Phi\right)\right\|_{WX} \le C\left\|M\left(f,\Phi\right)\right\|_{X}.
	\end{align*}
	Therefore, we have $\left\|f\right\|_{WH_{X}\left(\mathbb{R}^{n}\right)}\le C\left\|f\right\|_{H_{X}\left(\mathbb{R}^{n}\right)}$, which shows that $H_{X}\subset WH_{X}$. This finishes the proof of Lemma \ref{le:B-R H_w^s-WH_w^s}.
\end{proof}

\begin{lemma}\label{le:B-R H_w^s-WL_w^s}
	\textup{Let $s\in\left(0,\infty\right)$, $\omega\in A_{\infty}\left(\mathbb{R}^{n}\right)$ and $\delta>\frac{n-1}{2}$. Then there exists a positive constant $C$ such that, for any $f\in H_{\omega}^{s}\left(\mathbb{R}^{n}\right)$,}
	\begin{align}\label{eq:B-R H_w^s-WL_w^s}
		\left\|B_{1/\varepsilon}^{\delta}\left(f\right) \right\|_{WL_{\omega}^{s}\left(\mathbb{R}^{n}\right)} \le C\left\|f\right\|_{H_{\omega}^{s}\left(\mathbb{R}^{n}\right)},
	\end{align}
	\textup{where $WL_{\omega}^{s}\left(\mathbb{R}^{n}\right)$ is as in Definition \ref{de:WX} with $X$ replaced by $L_{\omega}^{s}\left(\mathbb{R}^{n}\right)$.}
\end{lemma}

\begin{proof}
	By \citep[Corollary 3.5({\rm \romannumeral2})]{ZhangTan}, we know that $\left\|B_{1/\varepsilon}^{\delta}\left(f\right) \right\|_{L_{\omega}^{s}\left(\mathbb{R}^{n}\right)} \le C\left\|f\right\|_{H_{\omega}^{s}\left(\mathbb{R}^{n}\right)}$ with $s\in\left(0,\infty\right)$. Therefore, to prove (\ref{eq:B-R H_w^s-WL_w^s}), it is only necessary to prove $\left\|B_{1/\varepsilon}^{\delta}\left(f\right) \right\|_{WL_{\omega}^{s}\left(\mathbb{R}^{n}\right)} \le C\left\|B_{1/\varepsilon}^{\delta}\left(f\right) \right\|_{L_{\omega}^{s}\left(\mathbb{R}^{n}\right)}$, which can be derived from $L_{\omega}^{s}\left(\mathbb{R}^{n}\right) \subset WL_{\omega}^{s}\left(\mathbb{R}^{n}\right)$. 
	
	Let $X$ be a ball quasi-Banach function space. From Remark \ref{re:X-subset-WX} ({\rm \romannumeral1}), we know that $X\subset WX$. Since $L_{\omega}^{s}\left(\mathbb{R}^{n}\right)$ is a ball quasi-Banach function space, it follows that $L_{\omega}^{s}\left(\mathbb{R}^{n}\right) \subset WL_{\omega}^{s}\left(\mathbb{R}^{n}\right)$. This finishes the proof of Lemma \ref{le:B-R H_w^s-WL_w^s}.
\end{proof}

\section{Proofs of main results}

{
\noindent \textit{\bf Proof of Theorem \ref{th:B-R Hx-WHx}}. Let $\theta$, $s$ and $d$ be as in Lemma \ref{le:H_X-H_atom equivalent quasi-norm} and $f\in H_{X}\left(\mathbb{R}^{n}\right)$. Then, by Lemma \ref{le:the atomic characterization for the Hardy type space}, we find that there exist a sequence $\left\{a_{j}\right\}_{j=1}^{\infty}$ of $\left(X,\infty,d\right)$-atoms supported, respectively, in a sequence $\left\{Q_{j}\right\}_{j=1}^{\infty}$ of cubes, and a sequence $\left\{\lambda_{j}\right\}_{j=1}^{\infty}$ of non-negative numbers, independent of $f$ but depending on $s$, such that
} 

\begin{align}\label{eq:th01-01}
	f=\sum_{j=1}^{\infty} \lambda_{j} a_{j}\quad {\rm in} \quad \mathcal{S}'\left(\mathbb{R}^{n}\right)
\end{align}
and
\begin{align}\label{eq:th01-02}
	\left\|\left\{\sum_{j=1}^{\infty}\left(\frac{\lambda_{j}}{\left\|\chi_{Q_j}\right\|_{X}}\right)^{s} \chi_{Q_j}\right\}^{1/s}\right\|_{X} \lesssim\|f\|_{H_{X}\left(\mathbb{R}^{n}\right)}.
\end{align}
From Lemma \ref{le:X-subset-the weighted Lebesgue space}, we deduce that there exists an $\epsilon\in\left(0,1\right)$ such that $X$ is continuously embedded into $L_{\omega}^{s}\left(\mathbb{R}^{n}\right)$ with $\omega:=\left[M\left(\chi_{B\left({\vec 0}_{n},1\right)}\right)\right]^{\epsilon}$, which, combined with (\ref{eq:th01-02}), implies that
\begin{align}\label{eq:th01-03}
	&\left\|\left\{\sum_{j=1}^{\infty} \left[\lambda_{j}\frac{\left\|\chi_{Q_{j}}\right\|_{L_{\omega}^{s}\left(\mathbb{R}^{n}\right)}}{\left\|\chi_{Q_{j}}\right\|_{X}} \frac{1}{\left\|\chi_{Q_{j}}\right\|_{L_{\omega}^{s}\left(\mathbb{R}^{n}\right)}} \right]^{s}\chi_{Q_{j}} \right\}^{1/s} \right\|_{L_{\omega}^{s}\left(\mathbb{R}^{n}\right)}\nonumber
	\\= &\left\|\left\{\sum_{j=1}^{\infty} \left(\frac{\lambda_{j}}{\left\|\chi_{Q_{j}}\right\|_{X}}\right)^{s}\chi_{Q_{j}} \right\}^{1/s}\right\|_{L_{\omega}^{s}\left(\mathbb{R}^{n}\right)}\nonumber\lesssim \left\|\left\{\sum_{j=1}^{\infty} \left(\frac{\lambda_{j}}{\left\|\chi_{Q_{j}}\right\|_{X}}\right)^{s}\chi_{Q_{j}} \right\}^{1/s}\right\|_{X}\nonumber
	\\\lesssim &\left\|f\right\|_{H_{X}\left(\mathbb{R}^{n}\right)}
\end{align}
Moreover, since for any $j\in\mathbb{N}$, $a_{j}$ is an $\left(X,\infty,d\right)$-atom, it follows that for any $j\in\mathbb{N}$, $\frac{\left\|\chi_{Q_{j}}\right\|_{X}}{\left\|\chi_{Q_{j}}\right\|_{L_{\omega}^{s}\left(\mathbb{R}^{n}\right)}} a_{j}$ is an $\left(L_{\omega}^{s}\left(\mathbb{R}^{n}\right),\infty,d\right)$-atom.
By Lemma \ref{le:characterization of all A_1 weights}, we know that $\omega\in A_{1}\left(\mathbb{R}^{n}\right)$, which, combined with \citep[Remarks 2.4(b) and 2.7(b)]{WangYangYang}, implies that $L_{\omega}^{s}\left(\mathbb{R}^{n}\right)$ satisfies all the assumptions of Lemmas \ref{le:the atomic characterization for the Hardy type space} and \ref{le:density}. Using Lemmas \ref{le:the atomic characterization for the Hardy type space}, \ref{le:density} and (\ref{eq:th01-03}), we conclude that
\begin{align}\label{eq:th01-04}
	\sum_{j=1}^{\infty} \left[\lambda_{j}\frac{\left\|\chi_{Q_{j}}\right\|_{L_{\omega}^{s}\left(\mathbb{R}^{n}\right)}}{\left\|\chi_{Q_{j}}\right\|_{X}} \right] \left[\frac{\left\|\chi_{Q_{j}}\right\|_{X}}{\left\|\chi_{Q_{j}}\right\|_{L_{\omega}^{s}\left(\mathbb{R}^{n}\right)}} a_{j} \right] = \sum_{j=1}^{\infty} \lambda_{j}a_{j}=f \quad {\rm in} \quad \mathcal{S}'\left(\mathbb{R}^{n}\right) \quad and \quad H_{\omega}^{s}\left(\mathbb{R}^{n}\right).
\end{align}
Furthermore, from Lemma \ref{le:B-R H_w^s-WH_w^s}, we know that $B_{1/\varepsilon}^{\delta}$ is bounded from $H_{\omega}^{s}\left(\mathbb{R}^{n}\right)$ to $WH_{\omega}^{s}\left(\mathbb{R}^{n}\right)$, and hence
\begin{align}\label{eq:th01-05}
	B_{1/\varepsilon}^{\delta}\left(f\right)=\phi_{\varepsilon}\ast f=\sum_{j=1}^{\infty} \lambda_{j} B_{1/\varepsilon}^{\delta}\left(a_{j}\right)\quad {\rm in}\quad WH_{\omega}^{s}\left(\mathbb{R}^{n}\right)\quad {\rm and}\quad \mathcal{S}'\left(\mathbb{R}^{n}\right),
\end{align}
where $\phi_{\varepsilon}$ denotes the kernel of $B_{1/\varepsilon}^{\delta}$. Suppose that $\phi\in C^{\infty}$, $\operatorname{supp}(\phi) \subset\{|x|<1\}$, and $\phi_{\varepsilon}(t)=\varepsilon^{-n} \phi\left(\frac{t}{\varepsilon}\right)$. Let $\Phi\in\mathcal{S}\left(\mathbb{R}^{n}\right)$ satisfy $\int_{\mathbb{R}^{n}} \Phi\left(x\right)\,dx\ne 0$. Then, to prove this theorem, by Assumption \ref{as:M is bounded on (WX)^{1/r}} and Lemma \ref{le:equivalent-WX}, we only need to show that for any $f\in H_{X}\left(\mathbb{R}^{n}\right)$,
\begin{align}\label{eq:th01-06}
	\left\|M\left(B_{1/\varepsilon}^{\delta}f,\Phi\right)\right\|_{WX} \lesssim \left\|f\right\|_{H_{X}\left(\mathbb{R}^{n}\right)}.
\end{align}
For any $\alpha\in\left(0,\infty\right)$, by (\ref{eq:th01-05}), Lemma \ref{le:quasi-norm on WX} ({\rm \romannumeral3}) and Remark \ref{re:X-subset-WX} ({\rm \romannumeral1}), we have
\begin{align}\label{eq:th01-07}
	&\alpha\left\|\chi_{\left\{x\in\mathbb{R}^{n}:M\left(B_{1/\varepsilon}^{\delta}f,\Phi\right)\left(x\right)>\alpha\right\}}\right\|_{X} \le \alpha\left\|\chi_{\left\{x\in\mathbb{R}^{n}:\sum_{j=1}^{\infty}\lambda_{j}M\left(B_{1/\varepsilon}^{\delta}a_{j},\Phi\right)\left(x\right)>\alpha\right\}}\right\|_{X}\nonumber
	\\\lesssim &\alpha\left\|\chi_{\left\{x\in\mathbb{R}^{n}:\sum_{j=1}^{\infty}\lambda_{j}M\left(B_{1/\varepsilon}^{\delta}a_{j},\Phi\right)\left(x\right)\chi_{4Q_{j}}\left(x\right)>\frac{\alpha}{2}\right\}}\right\|_{X} + \alpha\left\|\chi_{\left\{x\in\mathbb{R}^{n}:\sum_{j=1}^{\infty}\lambda_{j}M\left(B_{1/\varepsilon}^{\delta}a_{j},\Phi\right)\left(x\right)\chi_{\left(4Q_{j}\right)^{c}}\left(x\right)>\frac{\alpha}{2}\right\}}\right\|_{X}\nonumber
	\\\lesssim &\left\|\sum_{j=1}^{\infty} \lambda_{j}M\left(B_{1/\varepsilon}^{\delta}a_{j},\Phi\right)\chi_{4Q_{j}}\right\|_{X} + \alpha\left\|\chi_{\left\{x\in\mathbb{R}^{n}:\sum_{j=1}^{\infty}\lambda_{j}M\left(B_{1/\varepsilon}^{\delta}a_{j},\Phi\right)\left(x\right)\chi_{\left(4Q_{j}\right)^{c}}\left(x\right)>\frac{\alpha}{2}\right\}}\right\|_{X}\nonumber
	\\=:&I_{1}+I_{2}.
\end{align}
From this, to prove (\ref{eq:th01-06}), it is only necessary to prove $I_{1}\lesssim\left\|f\right\|_{H_{X}\left(\mathbb{R}^{n}\right)}$ and $I_{2}\lesssim\left\|f\right\|_{H_{X}\left(\mathbb{R}^{n}\right)}$, respectively.

We first estimate $I_{1}$. By \cite{CoifmanRochbergWeiss}, we know that if $\delta>\frac{n-1}{2}$, then the Bochner-Riesz means $B_{1/\varepsilon}^{\delta}$ is of type $\left(L^{p},L^{p}\right)$ with $1\le p\le\infty$. Notice that, for any $j\in\mathbb{N}$, $M\left(B_{1/\varepsilon}^{\delta}a_{j},\Phi\right)\lesssim M\left(B_{1/\varepsilon}^{\delta}a_{j}\right)$ and $a_{j}\in L^{q}\left(\mathbb{R}^{n}\right)$ with $q\in\left(1,\infty\right)$, by the fact that $M$ and $B_{1/\varepsilon}^{\delta}$ are both bounded on $L^{q}\left(\mathbb{R}^{n}\right)$ with $q\in\left(1,\infty\right)$ and the size condition of $a_{j}$, it follows that
\begin{align*}
	\left\|M\left(B_{1/\varepsilon}^{\delta}a_{j},\Phi\right)\chi_{4Q_{j}}\right\|_{L^{q}\left(\mathbb{R}^{n}\right)}&\le\left\|M\left(B_{1/\varepsilon}^{\delta}a_{j},\Phi\right)\right\|_{L^{q}\left(\mathbb{R}^{n}\right)}\lesssim \left\|M\left(B_{1/\varepsilon}^{\delta}a_{j}\right)\right\|_{L^{q}\left(\mathbb{R}^{n}\right)}
	\\&\lesssim \left\|B_{1/\varepsilon}^{\delta}a_{j}\right\|_{L^{q}\left(\mathbb{R}^{n}\right)}\lesssim \left\|a_{j}\right\|_{L^{q}\left(\mathbb{R}^{n}\right)}\lesssim \frac{\left|Q_{j}\right|^{1/q}}{\left\|\chi_{Q_{j}}\right\|_{X}},
\end{align*}
which, combined with Lemma \ref{le:atom2} and (\ref{eq:th01-02}), implies that
\begin{align}\label{eq:th01-08}
	I_{1}\lesssim\left\|\left\{\sum_{j=1}^{\infty} \left(\frac{\lambda_{j}}{\left\|\chi_{Q_{j}}\right\|_{X}}\right)^{s}\chi_{Q_{j}} \right\}^{1/s}\right\|_{X}\lesssim\left\|f\right\|_{H_{X}\left(\mathbb{R}^{n}\right)},
\end{align}
where the atom $a_{j}$ in Lemma \ref{le:atom2} is replaced by $M\left(B_{1/\varepsilon}^{\delta}a_{j},\Phi\right)\chi_{4Q_{j}}$.

To deal with the term $I_{2}$, for any $t\in\left(0,\infty\right)$, let $\phi^{\left(t\right)}=\phi\ast\Phi_{t}$ with $\Phi_{t}\left(\cdot\right)=t^{-n}\Phi\left(\cdot/t\right)$. Take $p_{0}\in\left(0,1\right)$ such that $\delta=\frac{n}{p_{0}}-\frac{n+1}{2}$, which satisfies $\delta>\frac{n-1}{2}$.

Without loss of generality, let $x_{j}$ denote the center of $Q_{j}$ and $r_{j}$ its side length. For any $x\in\left(4Q_{j}\right)^{c}$ and $t\in Q_{j}$, we have $\left|x-t\right|\ge \frac{3}{4}\left|x-x_{j}\right|$. If $\varepsilon\le \frac{3}{4}\left|x-x_{j}\right|$, since $\left|\frac{x-t}{\varepsilon}\right|\ge \frac{\left|x-t\right|}{\frac{3}{4}\left|x-x_{j}\right|}\ge 1$ and $\operatorname{supp}\left(\phi\right)\subset \left\{\left|x\right|<1\right\}$, then $\phi_{\varepsilon}\left(x-t\right)=\varepsilon^{-n}\phi\left(\frac{x-t}{\varepsilon}\right)=0$. Select $x\in\left(4Q_{j}\right)^{c}$ and $\varepsilon>\frac{3}{4}\left|x-x_{j}\right|$. Let $P_{\phi}\left(\cdot\frac{x}{\varepsilon}\right)$ denote the Taylor polynomial of $N$-th order of $\phi\left(t\right)$ at $t=\frac{x}{\varepsilon}$, where $N=\left[n\left(1/p_{0}-1\right)\right]$. If $t=\frac{x-y}{\varepsilon}$, then $\phi\left(\frac{x-y}{\varepsilon}\right)=P_{\phi}\left(-\frac{y}{\varepsilon}\right)+\frac{D^{N+1}\phi\left(\xi\right)}{\left(N+1\right)!}\left(-\frac{y}{\varepsilon}\right)^{N+1}$.

By Lemma \ref{le:phi^t}, for any $\left|x\right|\ge 4\left|y\right|$ and $\varepsilon>\frac{3}{4}\left|x\right|$, we have
\begin{align}\label{eq:th01-09}
	\varepsilon^{-n}\left|\phi^{\left(t\right)}\left(\frac{x-y}{\varepsilon}\right)-P_{\phi^{\left(t\right)}}\left(-\frac{y}{\varepsilon}\right)\right| \lesssim \frac{\left|y\right|^{N+1}}{\left|x\right|^{N+n+1}},
\end{align}
which combined with the vanishing moment condition of $a_{j}$, the H{\"o}lder inequality, and the size condition of $a_{j}$, we deduce that, for any $x\in\left(4Q_{j}\right)^{c}$ and $\varepsilon>\frac{3}{4}\left|x\right|$,
\begin{align*}
	M\left(B_{1/\varepsilon}^{\delta}a_{j},\Phi\right)\left(x\right)&=\sup\limits_{t\in\left(0,\infty\right)} \varepsilon^{-n}\left|\int_{\mathbb{R}^{n}} \left[\phi^{\left(t\right)}\left(\frac{x-y}{\varepsilon}\right)-P_{\phi^{\left(t\right)}}\left(-\frac{y}{\varepsilon}\right)\right]a_{j}\left(y\right)\,dy\right|
	\\&\lesssim \int_{Q_{j}} \frac{\left|y\right|^{N+1}}{\left|x-x_{j}\right|^{N+n+1}}\left|a_{j}\left(y\right)\right|\,dy \lesssim \frac{r_{j}^{N+1}}{\left|x-x_{j}\right|^{N+n+1}}\int_{Q_{j}} \left|a_{j}\left(y\right)\right|\,dy
	\\&\le \frac{r_{j}^{N+1}}{\left|x-x_{j}\right|^{N+n+1}}\left\|a_{j}\right\|_{L^{q}\left(\mathbb{R}^{n}\right)} \left|Q_{j}\right|^{1/q'} \lesssim \frac{r_{j}^{N+n+1}}{\left|x-x_{j}\right|^{N+n+1}} \frac{1}{\left\|\chi_{Q_{j}}\right\|_{X}},
\end{align*}
where $\frac{1}{q}+\frac{1}{q'}=1$. From this and the fact that $N+n+1>n/p_{0}=\delta+\frac{n+1}{2}$ and $r_{j}^{\delta+\frac{n+1}{2}}\left|x-x_{j}\right|^{-\left(\delta+\frac{n+1}{2}\right)} \sim \left[M\left(\chi_{Q_{j}}\right)\left(x\right)\right]^{\frac{n+1+2\delta}{2n}}$, we deduce that
\begin{align*}
	M\left(B_{1/\varepsilon}^{\delta}a_{j},\Phi\right)\left(x\right)&\lesssim \frac{r_{j}^{n/p_{0}}}{\left|x-x_{j}\right|^{n/p_{0}}}\frac{1}{\left\|\chi_{Q_{j}}\right\|_{X}}=\frac{r_{j}^{\delta+\frac{n+1}{2}}}{\left|x-x_{j}\right|^{\delta+\frac{n+1}{2}}}\frac{1}{\left\|\chi_{Q_{j}}\right\|_{X}}
	\\&\sim \left[M\left(\chi_{Q_{j}}\right)\left(x\right)\right]^{\frac{n+1+2\delta}{2n}}\frac{1}{\left\|\chi_{Q_{j}}\right\|_{X}}.
\end{align*}
This shows that for any $x\in\left(4Q_{j}\right)^{c}$,
\begin{align}\label{eq:th01-10}
	M\left(B_{1/\varepsilon}^{\delta}a_{j},\Phi\right)\left(x\right)\chi_{\left(4Q_{j}\right)^{c}}\left(x\right)\lesssim \left[M\left(\chi_{Q_{j}}\right)\left(x\right)\right]^{\frac{n+1+2\delta}{2n}}\frac{1}{\left\|\chi_{Q_{j}}\right\|_{X}}.
\end{align}
Therefore, by this, Definition \ref{de:convexity-concavity} ({\rm \romannumeral1}), (\ref{eq:weak-type Fefferman-Stein}), $\theta\in\left[\frac{2n}{n+1+2\delta},1\right)$, $0<\theta<s\le 1$ and (\ref{eq:th01-02}), we find that
\begin{align}\label{eq:th01-11}
	I_{2}&\lesssim \alpha\left\|\chi_{\left\{x\in\mathbb{R}^{n}:\sum_{j=1}^{\infty} \frac{\lambda_{j}}{\left\|\chi_{Q_{j}}\right\|_{X}} \left[M\left(\chi_{Q_{j}}\right)\left(x\right)\right]^{\frac{n+1+2\delta}{2n}}>\frac{\alpha}{2} \right\}} \right\|_{X}\nonumber
	\\&\lesssim \frac{\alpha}{2} \left\|\chi_{\left\{x\in\mathbb{R}^{n}:\left\{\sum_{j=1}^{\infty} \frac{\lambda_{j}}{\left\|\chi_{Q_{j}}\right\|_{X}} \left[M\left(\chi_{Q_{j}}\right)\left(x\right)\right]^{\frac{n+1+2\delta}{2n}}\right\}^{\frac{2n}{n+1+2\delta}}>\left(\frac{\alpha}{2}\right)^{\frac{2n}{n+1+2\delta}} \right\}} \right\|_{X^{\frac{n+1+2\delta}{2n}}}^{\frac{n+1+2\delta}{2n}}\nonumber
	\\&\lesssim \left\|\left(\sum_{j=1}^{\infty}\frac{\lambda_{j}}{\left\|\chi_{Q_{j}}\right\|_{X}}\chi_{Q_{j}}\right)^{\frac{2n}{n+1+2\delta}}\right\|_{X^{\frac{n+1+2\delta}{2n}}}^{\frac{n+1+2\delta}{2n}}\le \left\|\left\{\sum_{j=1}^{\infty}\left(\frac{\lambda_{j}\chi_{Q_{j}}}{\left\|\chi_{Q_{j}}\right\|_{X}}\right)^{\frac{2n}{n+1+2\delta}} \right\}^{\frac{n+1+2\delta}{2n}}\right\|_{X}\nonumber
	\\&\lesssim \left\|\left\{\sum_{j=1}^{\infty}\left(\frac{\lambda_{j}\chi_{Q_{j}}}{\left\|\chi_{Q_{j}}\right\|_{X}}\right)^{s} \right\}^{1/s}\right\|_{X}\lesssim \left\|f\right\|_{H_{X}\left(\mathbb{R}^{n}\right)}.
\end{align}
Finally, combining (\ref{eq:th01-08}) and (\ref{eq:th01-11}), we conclude that for any $\alpha\in\left(0,\infty\right)$,
\begin{align*}
	\alpha\left\|\chi_{\left\{x\in\mathbb{R}^{n}:M\left(B_{1/\varepsilon}^{\delta}f,\Phi\right)\left(x\right)>\alpha \right\}}\right\|_{X} \lesssim \left\|f\right\|_{H_{X}\left(\mathbb{R}^{n}\right)},
\end{align*}
namely, (\ref{eq:th01-06}) holds true. This finishes the proof of Theorem \ref{th:B-R Hx-WHx}.
$\hfill\square$\\

Next we will prove Theorem \ref{th:B-R Hx-WX}.\\
{
	\noindent \textit{\bf Proof of Theorem \ref{th:B-R Hx-WX}}.\quad Let $\theta$, $s$ and $d$ be as in Lemma \ref{le:H_X-H_atom equivalent quasi-norm} and $f\in H_{X}\left(\mathbb{R}^{n}\right)$. Then, by Lemma \ref{le:the atomic characterization for the Hardy type space}, we find that there exist a sequence $\left\{a_{j}\right\}_{j=1}^{\infty}$ of $\left(X,2,d\right)$-atoms supported, respectively, in a sequence $\left\{Q_{j}\right\}_{j=1}^{\infty}$ of cubes, and a sequence $\left\{\lambda_{j}\right\}_{j=1}^{\infty}$ of non-negative numbers, independent of $f$ but depending on $s$, such that (\ref{eq:th01-01}) and (\ref{eq:th01-02}) hold true.
}

By Lemma \ref{le:B-R H_w^s-WL_w^s}, we know that $B_{1/\varepsilon}^{\delta}$ is bounded from $H_{\omega}^{s}\left(\mathbb{R}^{n}\right)$ to $WL_{\omega}^{s}\left(\mathbb{R}^{n}\right)$. Therefore, from (\ref{eq:th01-03}) and (\ref{eq:th01-04}), combined with the $\left(H_{\omega}^{s},WL_{\omega}^{s}\right)$ boundedness of $B_{1/\varepsilon}^{\delta}$, we know that
\begin{align}\label{eq:th02-01}
	B_{1/\varepsilon}^{\delta}\left(f\right)=\phi_{\varepsilon}\ast f=\sum_{j=1}^{\infty} \lambda_{j} B_{1/\varepsilon}^{\delta}\left(a_{j}\right)\quad {\rm in}\quad WL_{\omega}^{s}\left(\mathbb{R}^{n}\right)\quad {\rm and}\quad \mathcal{S}'\left(\mathbb{R}^{n}\right),
\end{align}
where $\phi_{\varepsilon}$ denotes the kernel of $B_{1/\varepsilon}^{\delta}$. Suppose that $\phi\in C^{\infty}$, $\operatorname{supp}(\phi) \subset\{|x|<1\}$, and $\phi_{\varepsilon}(t)=\varepsilon^{-n} \phi\left(\frac{t}{\varepsilon}\right)$. Let $\Phi\in\mathcal{S}\left(\mathbb{R}^{n}\right)$ satisfy $\int_{\mathbb{R}^{n}} \Phi\left(x\right)\,dx\ne 0$. Then, to prove this theorem, by Definition \ref{de:WX}, we only need to show that for any $f\in H_{X}\left(\mathbb{R}^{n}\right)$ and any $\alpha\in\left(0,\infty\right)$,
\begin{align}\label{eq:th02-02}
	\alpha\left\|\chi_{\left\{x\in\mathbb{R}^{n}:\left|B_{1/\varepsilon}^{\delta}\left(f\right)\left(x\right)\right|>\alpha \right\}}\right\|_{X} \lesssim \left\|f\right\|_{H_{X}\left(\mathbb{R}^{n}\right)}.
\end{align}
For any $\alpha\in\left(0,\infty\right)$, by (\ref{eq:th02-01}), Lemma \ref{le:quasi-norm on WX} ({\rm \romannumeral3}) and Remark \ref{re:X-subset-WX} ({\rm \romannumeral1}), we have
\begin{align}\label{eq:th02-03}
	&\alpha\left\|\chi_{\left\{x\in\mathbb{R}^{n}:\left|B_{1/\varepsilon}^{\delta}\left(f\right)\left(x\right)\right|>\alpha\right\}}\right\|_{X} = \alpha\left\|\chi_{\left\{x\in\mathbb{R}^{n}:\left|\sum_{j=1}^{\infty}\lambda_{j}B_{1/\varepsilon}^{\delta}\left(a_{j}\right)\left(x\right)\right|>\alpha\right\}}\right\|_{X}\nonumber
	\\\lesssim &\alpha\left\|\chi_{\left\{x\in\mathbb{R}^{n}:\left|\sum_{j=1}^{\infty}\lambda_{j}B_{1/\varepsilon}^{\delta}\left(a_{j}\right)\left(x\right)\chi_{4Q_{j}}\left(x\right)\right|>\frac{\alpha}{2}\right\}}\right\|_{X} + \alpha\left\|\chi_{\left\{x\in\mathbb{R}^{n}:\left|\sum_{j=1}^{\infty}\lambda_{j}B_{1/\varepsilon}^{\delta}\left(a_{j}\right)\left(x\right)\chi_{\left(4Q_{j}\right)^{c}}\left(x\right)\right|>\frac{\alpha}{2}\right\}}\right\|_{X}\nonumber
	\\\lesssim &\left\|\sum_{j=1}^{\infty} \lambda_{j}B_{1/\varepsilon}^{\delta}\left(a_{j}\right)\chi_{4Q_{j}}\right\|_{X} + \alpha\left\|\chi_{\left\{x\in\mathbb{R}^{n}:\left|\sum_{j=1}^{\infty}\lambda_{j}B_{1/\varepsilon}^{\delta}\left(a_{j}\right)\left(x\right)\chi_{\left(4Q_{j}\right)^{c}}\left(x\right)\right|>\frac{\alpha}{2}\right\}}\right\|_{X}\nonumber
	\\=:&I_{1}+I_{2}.
\end{align}
From this, to prove (\ref{eq:th02-02}), it is only necessary to prove $I_{1}\lesssim\left\|f\right\|_{H_{X}\left(\mathbb{R}^{n}\right)}$ and $I_{2}\lesssim\left\|f\right\|_{H_{X}\left(\mathbb{R}^{n}\right)}$, respectively.

Notice that, for any $j\in\mathbb{N}$, $a_{j}\in L^{2}\left(\mathbb{R}^{n}\right)$. Since $B_{1/\varepsilon}^{\delta}$ is bounded on $L^{2}\left(\mathbb{R}^{n}\right)$ for $\delta>\frac{n-1}{2}$, it follows that
\begin{align*}
	\left\|B_{1/\varepsilon}^{\delta}\left(a_{j}\right)\chi_{4Q_{j}}\right\|_{L^{2}\left(\mathbb{R}^{n}\right)}\le \left\|B_{1/\varepsilon}^{\delta}\left(a_{j}\right)\right\|_{L^{2}\left(\mathbb{R}^{n}\right)}\lesssim \left\|a_{j}\right\|_{L^{2}\left(\mathbb{R}^{n}\right)}\lesssim \frac{\left|Q_{j}\right|^{1/2}}{\left\|\chi_{Q_{j}}\right\|_{X}},
\end{align*}
which, combined with Lemma \ref{le:atom2} and (\ref{eq:th01-02}), implies that
\begin{align}\label{eq:th02-04}
	I_{1}\lesssim\left\|\left\{\sum_{j=1}^{\infty}\left(\frac{\lambda_{j}}{\left\|\chi_{Q_{j}}\right\|_{X}}\right)^{s}\chi_{Q_{j}} \right\}^{1/s}\right\|_{X}\lesssim\left\|f\right\|_{H_{X}\left(\mathbb{R}^{n}\right)},
\end{align}
where the atom $a_{j}$ in Lemma \ref{le:atom2} is replaced by $B_{1/\varepsilon}^{\delta}\left(a_{j}\right)\chi_{4Q_{j}}$.

By Lemma \ref{le:phi(B-R kernel)}, take a $p_{0}\in\left(0,1\right)$ such that $\delta=\frac{n}{p_{0}}-\frac{n+1}{2}$, which satisfies $\delta>\frac{n-1}{2}$. Therefore, for all multi-indices $\beta$, we have
\begin{align}\label{eq:th02-05}
	\sup\limits_{x\in\mathbb{R}^{n}}\left(1+\left|x\right|\right)^{n/p_{0}} \left|D^{\beta}\phi\left(x\right)\right| \le C_{\left(\beta,n,p_{0}\right)}.
\end{align}
Select $x\in\left(4Q_{j}\right)^{c}$ and $\varepsilon>\frac{3}{4}\left|x-x_{j}\right|$. Let $P_{\phi}\left(\cdot\frac{x}{\varepsilon}\right)$ denote the Taylor polynomial of $N$-th order of $\phi\left(t\right)$ at $t=\frac{x}{\varepsilon}$, where $N=\left[n\left(1/p_{0}-1\right)\right]$. By the vanishing moment condition of $a_{j}$, the H{\"o}lder inequality and the size condition of $a_{j}$, we find that
\begin{align*}
	\left|B_{1/\varepsilon}^{\delta}\left(a_{j}\right)\left(x\right)\right|&=\varepsilon^{-n}\left|\int_{\mathbb{R}^{n}}\left[\phi\left(\frac{x-y}{\varepsilon}\right)-P_{\phi}\left(-\frac{y}{\varepsilon}\right)\right] a_{j}\left(y\right)\,dy \right|
	\\&\lesssim \varepsilon^{-n}\int_{\mathbb{R}^{n}}\frac{\left|D^{N+1}\phi\left(\xi\right)\right|}{\left(N+1\right)!} \left|\frac{y}{\varepsilon}\right|^{N+1}\left|a_{j}\left(y\right)\right|\,dy \lesssim \frac{r_{j}^{N+1}}{\left|x-x_{j}\right|^{N+n+1}}\int_{Q_{j}}\left|a_{j}\left(y\right)\right|\,dy\nonumber
	\\&\le\frac{r_{j}^{N+1}}{\left|x-x_{j}\right|^{N+n+1}}\left\|a_{j}\right\|_{L^{2}\left(\mathbb{R}^{n}\right)} \left|Q_{j}\right|^{1/2}\le \frac{r_{j}^{N+n+1}}{\left|x-x_{j}\right|^{N+n+1}}\frac{1}{\left\|\chi_{Q_{j}}\right\|_{X}},
\end{align*}
where $\left|D^{N+1}\phi\left(\xi\right)\right|\lesssim \left(1+\left|\xi\right|\right)^{-n/p_{0}}\lesssim\left|\xi\right|^{-n/p_{0}}\sim \left(\frac{\left|x-x_{j}\right|}{\varepsilon}\right)^{-n/p_{0}}$ and $0<N+n+1-\frac{n}{p_{0}}\le 1$. From the fact that $N+n+1>n/p_{0}=\delta+\frac{n+1}{2}$ and $r_{j}^{\delta+\frac{n+1}{2}}\left|x-x_{j}\right|^{-\left(\delta+\frac{n+1}{2}\right)}\sim \left[M\left(\chi_{Q_{j}}\right)\left(x\right)\right]^{\frac{n+1+2\delta}{2n}}$, we deduce that
\begin{align*}
	\left|B_{1/\varepsilon}^{\delta}\left(a_{j}\right)\left(x\right)\right|&\lesssim \frac{r_{j}^{n/p_{0}}}{\left|x-x_{j}\right|^{n/p_{0}}}\frac{1}{\left\|\chi_{Q_{j}}\right\|_{X}}=\frac{r_{j}^{\delta+\frac{n+1}{2}}}{\left|x-x_{j}\right|^{\delta+\frac{n+1}{2}}}\frac{1}{\left\|\chi_{Q_{j}}\right\|_{X}}
	\\&\sim\left[M\left(\chi_{Q_{j}}\right)\left(x\right)\right]^{\frac{n+1+2\delta}{2n}}\frac{1}{\left\|\chi_{Q_{j}}\right\|_{X}}.
\end{align*}
This shows that for any $x\in\left(4Q_{j}\right)^{c}$,
\begin{align}\label{eq:th02-06}
	\left|B_{1/\varepsilon}^{\delta}\left(a_{j}\right)\left(x\right)\right|\chi_{\left(4Q_{j}\right)^{c}}\left(x\right)\lesssim \left[M\left(\chi_{Q_{j}}\right)\left(x\right)\right]^{\frac{n+1+2\delta}{2n}}\frac{1}{\left\|\chi_{Q_{j}}\right\|_{X}}.
\end{align}
Therefore,
\begin{align*}
	I_{2}&\le \alpha\left\|\chi_{\left\{x\in\mathbb{R}^{n}:\sum_{j=1}^{\infty}\lambda_{j}\left|B_{1/\varepsilon}^{\delta}\left(a_{j}\right)\left(x\right)\right|\chi_{\left(4Q_{j}\right)^{c}}\left(x\right)>\frac{\alpha}{2} \right\}}\right\|_{X}
	\\&\lesssim\alpha\left\|\chi_{\left\{x\in\mathbb{R}^{n}:\sum_{j=1}^{\infty}\frac{\lambda_{j}}{\left\|\chi_{Q_{j}}\right\|_{X}}\left[M\left(\chi_{Q_{j}}\right)\left(x\right)\right]^{\frac{n+1+2\delta}{2n}}>\frac{\alpha}{2} \right\}}\right\|_{X}.
\end{align*}
As with (\ref{eq:th01-11}), by Definition \ref{de:convexity-concavity} ({\rm \romannumeral1}), (\ref{eq:weak-type Fefferman-Stein}), $\theta\in\left[\frac{2n}{n+1+2\delta},1\right)$, $0<\theta<s\le 1$ and (\ref{eq:th01-02}), we find that
\begin{align}\label{eq:th02-07}
	I_{2}\lesssim\left\|f\right\|_{H_{X}\left(\mathbb{R}^{n}\right)}.
\end{align}
Finally, combining (\ref{eq:th02-04}) and (\ref{eq:th02-07}), we conclude that (\ref{eq:th02-02}) holds true. This finishes the proof of Theorem \ref{th:B-R Hx-WX}.
$\hfill\square$\\

{
	\noindent \textit{
\bf Proof of Theorem \ref{th:MB-R Hx-WX}}.\quad Let $\theta$, $s$ and $d$ be as in Lemma \ref{le:H_X-H_atom equivalent quasi-norm} and $f\in H_{X}\left(\mathbb{R}^{n}\right)$. Then, by Lemma \ref{le:the atomic characterization for the Hardy type space}, we find that there exist a sequence $\left\{a_{j}\right\}_{j=1}^{\infty}$ of $\left(X,2,d\right)$-atoms supported, respectively, in a sequence $\left\{Q_{j}\right\}_{j=1}^{\infty}$ of cubes, and a sequence $\left\{\lambda_{j}\right\}_{j=1}^{\infty}$ of non-negative numbers, independent of $f$ but depending on $s$, such that (\ref{eq:th01-01}) and (\ref{eq:th01-02}) hold true. 
}

From (\ref{eq:th01-03}) and (\ref{eq:th01-04}), combined with the $\left(H_{\omega}^{s},WL_{\omega}^{s}\right)$ boundedness of $B_{\ast}^{\delta}$ \cite{Sato}, we know that
\begin{align}\label{eq:th03-01}
	B_{\ast}^{\delta}\left(f\right)=\sup\limits_{\varepsilon>0}\left|f\ast\phi_{\varepsilon} \right|\le\sum_{j=1}^{\infty} \lambda_{j} \phi_{+}^{\ast}\left(a_{j}\right)=\sum_{j=1}^{\infty}\lambda_{j}B_{\ast}^{\delta}\left(a_{j}\right)\quad {\rm in}\quad WL_{\omega}^{s}\left(\mathbb{R}^{n}\right)\quad {\rm and}\quad \mathcal{S}'\left(\mathbb{R}^{n}\right),
\end{align}
where $\phi_{\varepsilon}$ denotes the kernel of $B_{\ast}^{\delta}$. Suppose that $\phi\in C^{\infty}$, $\operatorname{supp}(\phi) \subset\{|x|<1\}$, and $\phi_{\varepsilon}(t)=\varepsilon^{-n} \phi\left(\frac{t}{\varepsilon}\right)$. Let $\Phi\in\mathcal{S}\left(\mathbb{R}^{n}\right)$ satisfy $\int_{\mathbb{R}^{n}} \Phi\left(x\right)\,dx\ne 0$. Then, to prove this theorem, by Definition \ref{de:WX}, we only need to show that for any $f\in H_{X}\left(\mathbb{R}^{n}\right)$ and any $\alpha\in\left(0,\infty\right)$,
\begin{align}\label{eq:th03-02}
	\alpha\left\|\chi_{\left\{x\in\mathbb{R}^{n}:B_{\ast}^{\delta}\left(f\right)\left(x\right)>\alpha \right\}}\right\|_{X} \lesssim \left\|f\right\|_{H_{X}\left(\mathbb{R}^{n}\right)}.
\end{align}
For any $\alpha\in\left(0,\infty\right)$, by (\ref{eq:th03-01}), Lemma \ref{le:quasi-norm on WX} ({\rm \romannumeral3}) and Remark \ref{re:X-subset-WX} ({\rm \romannumeral1}), we have
\begin{align}\label{eq:th03-03}
	&\alpha\left\|\chi_{\left\{x\in\mathbb{R}^{n}:B_{\ast}^{\delta}\left(f\right)\left(x\right)>\alpha\right\}}\right\|_{X} \le \alpha\left\|\chi_{\left\{x\in\mathbb{R}^{n}:\sum_{j=1}^{\infty}\lambda_{j}B_{\ast}^{\delta}\left(a_{j}\right)\left(x\right)>\alpha\right\}}\right\|_{X}\nonumber
	\\\lesssim &\alpha\left\|\chi_{\left\{x\in\mathbb{R}^{n}:\sum_{j=1}^{\infty}\lambda_{j}B_{\ast}^{\delta}\left(a_{j}\right)\left(x\right)\chi_{4Q_{j}}\left(x\right)>\frac{\alpha}{2}\right\}}\right\|_{X} + \alpha\left\|\chi_{\left\{x\in\mathbb{R}^{n}:\sum_{j=1}^{\infty}\lambda_{j}B_{\ast}^{\delta}\left(a_{j}\right)\left(x\right)\chi_{\left(4Q_{j}\right)^{c}}\left(x\right)>\frac{\alpha}{2}\right\}}\right\|_{X}\nonumber
	\\\lesssim &\left\|\sum_{j=1}^{\infty} \lambda_{j}B_{\ast}^{\delta}\left(a_{j}\right)\chi_{4Q_{j}}\right\|_{X} + \alpha\left\|\chi_{\left\{x\in\mathbb{R}^{n}:\sum_{j=1}^{\infty}\lambda_{j}B_{\ast}^{\delta}\left(a_{j}\right)\left(x\right)\chi_{\left(4Q_{j}\right)^{c}}\left(x\right)>\frac{\alpha}{2}\right\}}\right\|_{X}\nonumber
	\\=:&I_{1}+I_{2}.
\end{align}
From this, to prove (\ref{eq:th03-02}), it is only necessary to prove $I_{1}\lesssim\left\|f\right\|_{H_{X}\left(\mathbb{R}^{n}\right)}$ and $I_{2}\lesssim\left\|f\right\|_{H_{X}\left(\mathbb{R}^{n}\right)}$, respectively.

Notice that, for any $j\in\mathbb{N}$, $a_{j}\in L^{2}\left(\mathbb{R}^{n}\right)$ and $\phi_{+}^{\ast}\left(a_{j}\right)\lesssim M\left(a_{j}\right)$. Since $M$ is bounded on $L^{2}\left(\mathbb{R}^{n}\right)$, it follows that
\begin{align*}
	\left\|B_{\ast}^{\delta}\left(a_{j}\right)\chi_{4Q_{j}}\right\|_{L^{2}\left(\mathbb{R}^{n}\right)}&=\left\|\phi_{+}^{\ast}\left(a_{j}\right)\chi_{4Q_{j}}\right\|_{L^{2}\left(\mathbb{R}^{n}\right)} \lesssim \left\|M\left(a_{j}\right)\right\|_{L^{2}\left(\mathbb{R}^{n}\right)}
	\\&\lesssim \left\|a_{j}\right\|_{L^{2}\left(\mathbb{R}^{n}\right)}\lesssim \frac{\left|Q_{j}\right|^{1/2}}{\left\|\chi_{Q_{j}}\right\|_{X}},
\end{align*}
which, combined with Lemma \ref{le:atom2} and (\ref{eq:th01-02}), implies that
\begin{align}\label{eq:th03-04}
	I_{1}\lesssim\left\|\left\{\sum_{j=1}^{\infty}\left(\frac{\lambda_{j}}{\left\|\chi_{Q_{j}}\right\|_{X}}\right)^{s}\chi_{Q_{j}} \right\}^{1/s}\right\|_{X}\lesssim\left\|f\right\|_{H_{X}\left(\mathbb{R}^{n}\right)},
\end{align}
where the atom $a_{j}$ in Lemma \ref{le:atom2} is replaced by $B_{\ast}^{\delta}\left(a_{j}\right)\chi_{4Q_{j}}$.

By Lemma \ref{le:phi(B-R kernel)}, take a $p_{0}\in\left(0,1\right)$ such that $\delta=\frac{n}{p_{0}}-\frac{n+1}{2}$, which satisfies $\delta>\frac{n-1}{2}$. Therefore, for all multi-indices $\beta$, we obtain (\ref{eq:th02-05}).

Select $x\in\left(4Q_{j}\right)^{c}$ and $\varepsilon>\frac{3}{4}\left|x-x_{j}\right|$. Let $P_{\phi}\left(\cdot\frac{x}{\varepsilon}\right)$ denote the Taylor polynomial of $N$-th order of $\phi\left(t\right)$ at $t=\frac{x}{\varepsilon}$, where $N=\left[n\left(1/p_{0}-1\right)\right]$. By the vanishing moment condition of $a_{j}$, the H{\"o}lder inequality and the size condition of $a_{j}$, we find that
\begin{align*}
	B_{\ast}^{\delta}\left(a_{j}\right)\left(x\right)&=\sup\limits_{\varepsilon>0}\varepsilon^{-n}\left|\int_{\mathbb{R}^{n}}\left[\phi\left(\frac{x-y}{\varepsilon}\right)-P_{\phi}\left(-\frac{y}{\varepsilon}\right)\right] a_{j}\left(y\right)\,dy \right|
	\\&\lesssim \sup\limits_{\varepsilon>0}\varepsilon^{-n}\int_{\mathbb{R}^{n}}\frac{\left|D^{N+1}\phi\left(\xi\right)\right|}{\left(N+1\right)!} \left|\frac{y}{\varepsilon}\right|^{N+1}\left|a_{j}\left(y\right)\right|\,dy \lesssim \frac{r_{j}^{N+1}}{\left|x-x_{j}\right|^{N+n+1}}\int_{Q_{j}}\left|a_{j}\left(y\right)\right|\,dy\nonumber
	\\&\le\frac{r_{j}^{N+1}}{\left|x-x_{j}\right|^{N+n+1}}\left\|a_{j}\right\|_{L^{2}\left(\mathbb{R}^{n}\right)} \left|Q_{j}\right|^{1/2}\le \frac{r_{j}^{N+n+1}}{\left|x-x_{j}\right|^{N+n+1}}\frac{1}{\left\|\chi_{Q_{j}}\right\|_{X}}.
\end{align*}
From the fact that $N+n+1>n/p_{0}=\delta+\frac{n+1}{2}$ and $r_{j}^{\delta+\frac{n+1}{2}}\left|x-x_{j}\right|^{-\left(\delta+\frac{n+1}{2}\right)}\sim \left[M\left(\chi_{Q_{j}}\right)\left(x\right)\right]^{\frac{n+1+2\delta}{2n}}$, we deduce that
\begin{align*}
	B_{\ast}^{\delta}\left(a_{j}\right)\left(x\right)&\lesssim \frac{r_{j}^{n/p_{0}}}{\left|x-x_{j}\right|^{n/p_{0}}}\frac{1}{\left\|\chi_{Q_{j}}\right\|_{X}}=\frac{r_{j}^{\delta+\frac{n+1}{2}}}{\left|x-x_{j}\right|^{\delta+\frac{n+1}{2}}}\frac{1}{\left\|\chi_{Q_{j}}\right\|_{X}}
	\\&\sim\left[M\left(\chi_{Q_{j}}\right)\left(x\right)\right]^{\frac{n+1+2\delta}{2n}}\frac{1}{\left\|\chi_{Q_{j}}\right\|_{X}}.
\end{align*}
This shows that for any $x\in\left(4Q_{j}\right)^{c}$,
\begin{align}\label{eq:th03-06}
	B_{\ast}^{\delta}\left(a_{j}\right)\left(x\right)\chi_{\left(4Q_{j}\right)^{c}}\left(x\right)\lesssim \left[M\left(\chi_{Q_{j}}\right)\left(x\right)\right]^{\frac{n+1+2\delta}{2n}}\frac{1}{\left\|\chi_{Q_{j}}\right\|_{X}}.
\end{align}
Therefore,
\begin{align*}
	I_{2}\lesssim\alpha\left\|\chi_{\left\{x\in\mathbb{R}^{n}:\sum_{j=1}^{\infty}\frac{\lambda_{j}}{\left\|\chi_{Q_{j}}\right\|_{X}}\left[M\left(\chi_{Q_{j}}\right)\left(x\right)\right]^{\frac{n+1+2\delta}{2n}}>\frac{\alpha}{2} \right\}}\right\|_{X}.
\end{align*}
As with (\ref{eq:th01-11}), by Definition \ref{de:convexity-concavity} ({\rm \romannumeral1}), (\ref{eq:weak-type Fefferman-Stein}), $\theta\in\left[\frac{2n}{n+1+2\delta},1\right)$, $0<\theta<s\le 1$ and (\ref{eq:th01-02}), we find that
\begin{align}\label{eq:th03-07}
	I_{2}\lesssim\left\|f\right\|_{H_{X}\left(\mathbb{R}^{n}\right)}.
\end{align}
Finally, combining (\ref{eq:th03-04}) and (\ref{eq:th03-07}), we conclude that (\ref{eq:th03-02}) holds true. This finishes the proof of Theorem \ref{th:MB-R Hx-WX}.
$\hfill\square$

\section{Applications}
In this section, we apply all above results to the following examples of ball quasi-Banach function spaces, namely, weighted Lebesgue spaces, Herz spaces, Lorentz spaces, variable Lebesgue spaces and Morrey spaces.

\subsection{Weighted Lebesgue spaces.}
The concept of the weighted Lebesgue space is as in Definition \ref{de:the weighted Lebesgue space}. 

Let $X$ be a ball quasi-Banach function space. Let $p\in\left(0,\infty\right)$ and $\omega\in A_{\infty}\left(\mathbb{R}^{n}\right)$. If $X:=L_{\omega}^{p}\left(\mathbb{R}^{n}\right)$, then $WX:=WL_{\omega}^{p}\left(\mathbb{R}^{n}\right)$ is the weighted weak Lebesgue space, $H_{X}\left(\mathbb{R}^{n}\right):=H_{\omega}^{p}\left(\mathbb{R}^{n}\right)$ is the weighted Hardy space, and $WH_{X}\left(\mathbb{R}^{n}\right):=WH_{\omega}^{p}\left(\mathbb{R}^{n}\right)$ is the weighted weak Hardy space. The Definitions of weighted weak Lebesgue spaces, weighted Hardy spaces and weighted weak Hardy spaces are as in Definition \ref{de:WX}, \ref{de:the Hardy type space} and \ref{de:the weak Hardy-type space} with $X$ replaced by $L_{\omega}^{p}\left(\mathbb{R}^{n}\right)$.

Moreover, the weighted Lebesgue space $L_{\omega}^{p}\left(\mathbb{R}^{n}\right)$ satisfies Assumptions \ref{as:the boundedness of M on X^1/p}, \ref{as:M is bounded on (WX)^{1/r}} and \ref{as:(X^{1/s})'}, respectively, in Lemmas \ref{le:Weighted Lebesgue-as-X^1/s}, \ref{le:Weighted Lebesgue-as-(WX)^1/s} and \ref{le:Weighted Lebesgue-as-(X^{1/r})'}. The following lemma shows that the Fefferman-Stein vector-valued maximal inequalities for $L_{\omega}^{p}\left(\mathbb{R}^{n}\right)$ hold true.

\begin{lemma}\label{le:Weighted Lebesgue-as-X^1/s}
	\textup{\citep[Corollary 4.3]{Ho2016} Let $1<p<\infty$ and $\omega\in A_{p}$. Assume that $r\in\left(1,\infty\right)$ and $s\in\left(0,p\right)$. Then there exists a positive constant $C$ such that for any $\left\{f_{j}\right\}_{j=1}^{\infty}\subset \mathscr{M}\left(\mathbb{R}^{n}\right)$,}
	\begin{align*}
		\left\|\left\{\sum_{j=1}^{\infty}\left[M\left(f_{j}\right)\right]^{r}\right\}^{1/r}\right\|_{\left[L_{\omega}^{p}\left(\mathbb{R}^{n}\right)\right]^{1/s}} \le C\left\|\left\{\sum_{j=1}^{\infty}\left|f_{j}\right|^{r}\right\}^{1/r}\right\|_{\left[L_{\omega}^{p}\left(\mathbb{R}^{n}\right)\right]^{1/s}}.
	\end{align*}
\end{lemma}

Furthermore, the following lemma shows that the Fefferman--Stein vector-valued maximal inequalities for $WL_{\omega}^{p}\left(\mathbb{R}^{n}\right)$ hold true.

\begin{lemma}\label{le:Weighted Lebesgue-as-(WX)^1/s}
	\textup{(see \citep[Lemma 6.1.6]{LiYangHuang} or \citep[Threorem 4.4]{SunYangYuan}) Let $1<p<\infty$ and $\omega\in A_{p}$. Assume that $r\in\left(1,\infty\right)$ and $s\in\left(0,p\right)$. Then there exists a positive constant $C$ such that for any $\left\{f_{j}\right\}_{j=1}^{\infty}\subset \mathscr{M}\left(\mathbb{R}^{n}\right)$,}
	\begin{align*}
		\left\|\left\{\sum_{j=1}^{\infty}\left[M\left(f_{j}\right)\right]^{r}\right\}^{1/r}\right\|_{\left[WL_{\omega}^{p}\left(\mathbb{R}^{n}\right)\right]^{1/s}} \le C\left\|\left\{\sum_{j=1}^{\infty}\left|f_{j}\right|^{r}\right\}^{1/r}\right\|_{\left[WL_{\omega}^{p}\left(\mathbb{R}^{n}\right)\right]^{1/s}}.
	\end{align*}
\end{lemma}

The following lemma shows that the weighted Lebesgue space $L_{\omega}^{p}\left(\mathbb{R}^{n}\right)$ satisfies Assumptions \ref{as:(X^{1/s})'}.

\begin{lemma}\label{le:Weighted Lebesgue-as-(X^{1/r})'}
	\textup{\citep[Remark 2.7(b)]{WangYangYang} Let $0<p<\infty$ and $\omega\in A_{\infty}\left(\mathbb{R}^{n}\right)$. Then, for any $r\in\left(0,\min{\left\{1,p\right\}}\right)$, $\omega\in A_{p/r}\left(\mathbb{R}^{n}\right)$ and $s\in\left(\max{\left\{1,p\right\}},\infty\right]$ large enough such that $\omega^{1-\left(p/r\right)'}\in A_{\left(p/r\right)'/\left(s/r\right)'}\left(\mathbb{R}^{n}\right)$, there exists a positive constant $C$ such that for any $\left\{f\right\}_{j=1}^{\infty}\subset \mathscr{M}\left(\mathbb{R}^{n}\right)$,}
	\begin{align*}
		\left\|M^{\left((s/r)'\right)}(f)\right\|_{\left(\left[L_{\omega}^{p}\left(\mathbb{R}^{n}\right)\right]^{1/r}\right)'} \le C\|f\|_{\left(\left[L_{\omega}^{p}\left(\mathbb{R}^{n}\right)\right]^{1/r}\right)'},
	\end{align*}
	\textup{where $\left(\left[L_{\omega}^{p}\left(\mathbb{R}^{n}\right)\right]^{1/r}\right)'$ is as in (\ref{eq:X'}) with $X:=\left[L_{\omega}^{p}\left(\mathbb{R}^{n}\right)\right]^{1/r}$.}
\end{lemma}

For the ball quasi-Banach function space $X:=L_{\omega}^{p}\left(\mathbb{R}^{n}\right)$, to apply Theorems \ref{th:B-R Hx-WHx}, \ref{th:B-R Hx-WX} and \ref{th:MB-R Hx-WX} to weighted Lebesgue spaces, we need the following weak-type Fefferman-Stein vector-valued inequality of the Hardy-Littlewood maximal operator $M$ in (\ref{eq:the Hardy-Littlewood maximal operator}) from $X^{1/s}$ to $WX^{1/s}$ with $s\in\left(0,\infty\right)$. In fact, the following proposition shows that (\ref{eq:weak-type Fefferman-Stein}) holds true for all ball quasi-Banach function spaces.

\begin{proposition}\label{pro:X}
	\textup{Let $r\in\left(1,\infty\right)$, $s\in\left(0,\infty\right)$ and $X$ be a ball quasi-Banach function space. Assume that $X^{1/s}$ is a ball Banach function space, there exists a $q_{0}\in\left[1,\infty\right)$ such that $X^{1/\left(sq_{0}\right)}$ is a Banach function space and $M$ is bounded on $\left(X^{1/\left(sq_{0}\right)}\right)'$. Then there exists a positive constant $C$ such that, for any sequence $\left\{f_{j}\right\}_{j=1}^{\infty}\subset \mathscr{M}\left(\mathbb{R}^{n}\right)$ and $\alpha\in\left(0,\infty\right)$,}
	\begin{align}\label{eq:pro-X}
		\alpha\left\|\chi_{\left\{x\in\mathbb{R}^{n}:\left\{\sum_{j=1}^{\infty}\left[M\left(f_{j}\right)\left(x\right)\right]^{r}\right\}^{1/r}>\alpha \right\}}\right\|_{X^{1/s}}\le C\left\|\left\{\sum_{j=1}^{\infty}\left|f_{j}\right|^{r}\right\}^{1/r}\right\|_{X^{1/s}}.
	\end{align}
\end{proposition}

Before proving Proposition \ref{pro:X}, we need to recall the following two lemmas. Lemma \ref{le:extrapolation theorem} is an extrapolation theorem (see, for example, \citep[Lemma 7.34]{ZhangYangYuan} and \citep[Theorem 4.6]{CruzMartell}). Furthermore, lemma \ref{le:the weak-type weighted Fefferman--Stein} is the weak-type weighted Fefferman--Stein vector-valued inequality of the Hardy--Littlewood maximal operator $M$ from \citep[Theorem 3.1(a)]{AndersenJohn}.

\begin{lemma}\label{le:extrapolation theorem}
	\textup{Let $X$ be a ball Banach function space and $p_{0}\in\left(0,\infty\right)$. Let $\mathcal{F}$ be the set of all pairs of non-negative measurable functions $\left(F,G\right)$ such that for any given $\omega\in A_{1}\left(\mathbb{R}^{n}\right)$,}
	\begin{align*}
		\int_{\mathbb{R}^{n}}\left[F\left(x\right)\right]^{p_{0}}\omega\left(x\right)\,dx\le C_{\left(p_{0},\left[\omega\right]_{A_{1}\left(\mathbb{R}^{n}\right)}\right)} \int_{\mathbb{R}^{n}}\left[G\left(x\right)\right]^{p_{0}}\omega\left(x\right)\,dx,
	\end{align*}
	\textup{where $C_{\left(p_{0},\left[\omega\right]_{A_{1}\left(\mathbb{R}^{n}\right)}\right)}$ is a positive constant independent of $\left(F,G\right)$, but depends on $p_{0}$ and $\left[\omega\right]_{A_{1}\left(\mathbb{R}^{n}\right)}$. Assume that there exists a $q_{0}\in\left[p_{0},\infty\right)$ such that $X^{1/q_{0}}$ is a Banach function space and $M$ as in (\ref{eq:the Hardy-Littlewood maximal operator}) is bounded on $\left(X^{1/q_{0}}\right)'$. Then there exists a positive constant $C$ such that for any $\left(F,G\right)\in\mathcal{F}$,}
	\begin{align*}
		\left\|F\right\|_{X}\le C\left\|G\right\|_{X},
	\end{align*} 
	\textup{furthermore, for every $p$ with $p_{0}/q_{0}\le p<\infty$,}
	\begin{align*}
		\left\|F\right\|_{X^{p}}\le C\left\|G\right\|_{X^{p}}.
	\end{align*}
\end{lemma}

\begin{lemma}\label{le:the weak-type weighted Fefferman--Stein}
	\textup{Let $r\in\left(1,\infty\right)$ and $\omega\in A_{1}\left(\mathbb{R}^{n}\right)$. Then there exists a positive constant $C$, depending on $n$, $r$ and $\left[\omega\right]_{A_{1}\left(\mathbb{R}^{n}\right)}$, such that for all $\alpha\in\left(0,\infty\right)$ and $\left\{f_{j}\right\}_{j=1}^{\infty}\subset\mathscr{M}\left(\mathbb{R}^{n}\right)$,}
	\begin{align*}
		\alpha\omega\left(\left\{x\in\mathbb{R}^{n}:\left\{\sum_{j=1}^{\infty}\left[M\left(f_{j}\right)\left(x\right)\right]^{r}\right\}^{\frac{1}{r}}>\alpha\right\}\right) \le C\int_{\mathbb{R}^{n}}\left\{\sum_{j=1}^{\infty}\left|f_{j}\left(x\right)\right|^{r}\right\}^{\frac{1}{r}}\omega\left(x\right)\,dx.
	\end{align*}
\end{lemma}

{
	\noindent \textit{\bf Proof of Proposition \ref{pro:X}}.\quad For any $r\in\left(1,\infty\right)$, $\alpha\left(0,\infty\right)$ and any sequence $\left\{f_{j}\right\}_{j=1}^{\infty}\subset\mathscr{M}\left(\mathbb{R}^{n}\right)$, let $\mathcal{F}_{\alpha}$ be the set of all pairs $\left(F_{\alpha},G\right)$, where, for all $x\in\mathbb{R}^{n}$,} 
\begin{align*}
	F_{\alpha}\left(x\right)=\alpha\chi_{\left\{x\in\mathbb{R}^{n}:\left\{\sum_{j=1}^{\infty}\left[M\left(f_{j}\right)\left(x\right)\right]^{r}\right\}^{\frac{1}{r}}>\alpha\right\}}\left(x\right)\quad {\rm and}\quad G\left(x\right)=\left\{\sum_{j=1}^{\infty}\left|f_{j}\left(x\right)\right|^{r}\right\}^{\frac{1}{r}}.
\end{align*}
Then, by Lemma \ref{le:the weak-type weighted Fefferman--Stein}, we know that, for every $\omega\in A_{1}\left(\mathbb{R}^{n}\right)$,
\begin{align*}
	\int_{\mathbb{R}^{n}}F_{\alpha}\left(x\right)\omega\left(x\right)\,dx&=\alpha\omega\left(\left\{x\in\mathbb{R}^{n}:\left\{\sum_{j=1}^{\infty}\left[M\left(f_{j}\right)\left(x\right)\right]^{r}\right\}^{\frac{1}{r}}>\alpha\right\}\right)
	\\&\lesssim\int_{\mathbb{R}^{n}}\left\{\sum_{j=1}^{\infty}\left|f_{j}\left(x\right)\right|^{r}\right\}^{\frac{1}{r}}\omega\left(x\right)\,dx=\int_{\mathbb{R}^{n}}G\left(x\right)\omega\left(x\right)\,dx.
\end{align*}
Thus, by this and the fact that $M$ is bounded on $\left(X^{1/\left(sq_{0}\right)}\right)'$, applying Lemma \ref{le:extrapolation theorem} with $p_{0}:=1$, $\mathcal{F}=\mathcal{F}_{\alpha}$ and $X$ in Lemma \ref{le:extrapolation theorem} replaced by $X^{1/s}$ in this proposition, we conclude that for all $\alpha\in\left(0,\infty\right)$,
\begin{align*}
	\left\|F_{\alpha}\right\|_{X^{1/s}}\le C\left\|G\right\|_{X^{1/s}},
\end{align*}
namely,
\begin{align*}
	\alpha\left\|\chi_{\left\{x\in\mathbb{R}^{n}:\left\{\sum_{j=1}^{\infty}\left[M\left(f_{j}\right)\left(x\right)\right]^{r}\right\}^{1/r}>\alpha \right\}}\right\|_{X^{1/s}}\le C\left\|\left\{\sum_{j=1}^{\infty}\left|f_{j}\right|^{r}\right\}^{1/r}\right\|_{X^{1/s}},
\end{align*}
which completes the proof of Proposition \ref{pro:X}.
$\hfill\square$

Applying Lemmas \ref{le:Weighted Lebesgue-as-X^1/s}, \ref{le:Weighted Lebesgue-as-(WX)^1/s} and \ref{le:Weighted Lebesgue-as-(X^{1/r})'}, Proposition \ref{pro:X}, Theorems \ref{th:B-R Hx-WHx}, \ref{th:B-R Hx-WX} and \ref{th:MB-R Hx-WX}, we immediately obtain the following boundedness of Bochner-Riesz means and the maximal Bochner-Riesz means, respectively, as follows.

\begin{theorem}\label{th:Weighted Lebesgue spaces-B-R Hx-WHx}
	\textup{Let $0<p<\infty$, $\omega\in A_{\infty}\left(\mathbb{R}^{n}\right)$ and $\delta>\frac{n-1}{2}$. Let $B_{1/\varepsilon}^{\delta}$ be a Bochner--Riesz means with $\delta$ order on $\mathbb{R}^{n}$. If $\theta\in\left[\frac{2n}{n+1+2\delta},1\right]$, then $B_{1/\varepsilon}^{\delta}$ has a unique extension on $H_{\omega}^{p}\left(\mathbb{R}^{n}\right)$. Moreover, there exists a positive constant $C$ such that for any $f\in H_{\omega}^{p}\left(\mathbb{R}^{n}\right)$,}
	\begin{align*}
		\left\|B_{1/\varepsilon}^{\delta}\left(f\right)\right\|_{WH_{\omega}^{p}\left(\mathbb{R}^{n}\right)} \le C\left\|f\right\|_{H_{\omega}^{p}\left(\mathbb{R}^{n}\right)}.
	\end{align*}
\end{theorem}

\begin{theorem}\label{th:Weighted Lebesgue spaces-B-R Hx-WX}
	\textup{Let $0<p<\infty$, $\omega\in A_{\infty}\left(\mathbb{R}^{n}\right)$ and $\delta>\frac{n-1}{2}$. Let $B_{1/\varepsilon}^{\delta}$ be a Bochner-Riesz means with $\delta$ order on $\mathbb{R}^{n}$. If $\theta\in\left[\frac{2n}{n+1+2\delta},1\right]$, then $B_{1/\varepsilon}^{\delta}$ has a unique extension on $H_{\omega}^{p}\left(\mathbb{R}^{n}\right)$. Moreover, there exists a positive constant $C$ such that for any $f\in H_{\omega}^{p}\left(\mathbb{R}^{n}\right)$,}
	\begin{align*}
		\left\|B_{1/\varepsilon}^{\delta}\left(f\right)\right\|_{WL_{\omega}^{p}\left(\mathbb{R}^{n}\right)} \le C\left\|f\right\|_{H_{\omega}^{p}\left(\mathbb{R}^{n}\right)}.
	\end{align*}
\end{theorem}

\begin{theorem}\label{th:Weighted Lebesgue spaces-MB-R Hx-WX}
	\textup{Let $0<p<\infty$, $\omega\in A_{\infty}\left(\mathbb{R}^{n}\right)$ and $\delta>\frac{n-1}{2}$. Let $B_{\ast}^{\delta}$ be the maximal Bochner--Riesz means with $\delta$ order on $\mathbb{R}^{n}$. If $\theta\in\left[\frac{2n}{n+1+2\delta},1\right]$, then $B_{\ast}^{\delta}$ has a unique extension on $H_{\omega}^{p}\left(\mathbb{R}^{n}\right)$. Moreover, there exists a positive constant $C$ such that for any $f\in H_{\omega}^{p}\left(\mathbb{R}^{n}\right)$,}
	\begin{align*}
		\left\|B_{\ast}^{\delta}\left(f\right)\right\|_{WL_{\omega}^{p}\left(\mathbb{R}^{n}\right)} \le C\left\|f\right\|_{H_{\omega}^{p}\left(\mathbb{R}^{n}\right)}.
	\end{align*}
\end{theorem}

\subsection{Herz spaces.}
We begin with recalling the notion of Herz spaces.

\begin{definition}\label{de:Herz spaces}
	\textup{For the cube $Q\left({\vec 0}_{n},1\right)$ and $j\in\mathbb{N}$,}
	\begin{align*}
		S_{j}\left(Q\left({\vec 0}_{n},1\right)\right):=Q\left({\vec 0}_{n},2^{j+1}\right)\setminus Q\left({\vec 0}_{n},2^{j}\right),
	\end{align*}
	\textup{where, ${\vec 0}_{n}$ denotes the origin of $\mathbb{R}^{n}$. Then, for any $\alpha\in\mathbb{R}$ and $p,q\in\left(0,\infty\right]$, the Herz space $K_{p,q}^{\alpha}\left(\mathbb{R}^{n}\right)$ is defined to be the set of all measurable functions $f$ on $\mathbb{R}^{n}$ satisfying}
	\begin{align*}
		\left\|f\right\|_{\mathcal{K}_{p,q}^{\alpha}\left(\mathbb{R}^{n}\right)}:=\left\|\chi_{Q\left({\vec 0}_{n},2\right)}f\right\|_{L^{p}\left(\mathbb{R}^{n}\right)} + \left\{\sum_{j=1}^{\infty} \left[2^{j\alpha}\left\|\chi_{S_{j}\left(Q\left({\vec 0}_{n},1\right)\right)}f\right\|_{L^{p}\left(\mathbb{R}^{n}\right)} \right]^{q} \right\}^{\frac{1}{q}} <\infty.
	\end{align*}
\end{definition}

Let $X$ be a ball quasi-Banach function space. Let $\alpha\in\mathbb{R}$ and $p,q\in\left(0,\infty\right]$. If $X:=K_{p,q}^{\alpha}\left(\mathbb{R}^{n}\right)$, then $WX:=WK_{p,q}^{\alpha}\left(\mathbb{R}^{n}\right)$ is the weak Herz space, $H_{X}\left(\mathbb{R}^{n}\right):=HK_{p,q}^{\alpha}\left(\mathbb{R}^{n}\right)$ is the Herz--Hardy space, and $WH_{X}\left(\mathbb{R}^{n}\right):=WHK_{p,q}^{\alpha}\left(\mathbb{R}^{n}\right)$ is the weak Herz-Hardy space. The Definitions of weak Herz spaces, Herz--Hardy spaces and weak Herz--Hardy spaces are as in Definition \ref{de:WX}, \ref{de:the Hardy type space} and \ref{de:the weak Hardy-type space} with $X$ replaced by $K_{p,q}^{\alpha}\left(\mathbb{R}^{n}\right)$.

The Herz--Hardy space $K_{p,q}^{\alpha}\left(\mathbb{R}^{n}\right)$ was introduced by Chen and Lau \cite{ChenLau} and Garc{\'i}a-Cuerva \cite{Garcia-Cuerva} in 1989, and its theories were later further developed by Garc{\'i}a-Cuerva and Herrero \cite{Garcia-CuervaHerrero}. Furthermore, Lu and Yang \cite{LuYangHerz} studied weighted Herz--Hardy spaces.

Moreover, the Herz space $K_{p,q}^{\alpha}\left(\mathbb{R}^{n}\right)$ satisfies Assumptions \ref{as:the boundedness of M on X^1/p}, \ref{as:M is bounded on (WX)^{1/r}} and \ref{as:(X^{1/s})'}, respectively, in Lemmas \ref{le:Herz-as-X^1/s}, \ref{le:Herz-as-(WX)^1/s} and \ref{le:Herz-as-(X^{1/r})'}. The following lemma shows that the Fefferman-Stein vector-valued maximal inequalities for $K_{p,q}^{\alpha}\left(\mathbb{R}^{n}\right)$ hold true.

\begin{lemma}\label{le:Herz-as-X^1/s}
	\textup{\citep[Corollary 4.5]{Izuki} Let $p\in\left(1,\infty\right)$, $q\in\left(0,\infty\right)$ and $\alpha\in\left(-n/p,\infty\right)$. Assume that $r\in\left(1,\infty\right)$ and $s\in\left(0,\min{\left\{p,\left[\alpha/n+1/p\right]^{-1}\right\}}\right)$. Then there exists a positive constant $C$ such that for any $\left\{f_{j}\right\}_{j=1}^{\infty}\subset \mathscr{M}\left(\mathbb{R}^{n}\right)$,}
	\begin{align*}
		\left\|\left\{\sum_{j=1}^{\infty}\left[M\left(f_{j}\right)\right]^{r}\right\}^{1/r}\right\|_{\left[K_{p,q}^{\alpha}\left(\mathbb{R}^{n}\right)\right]^{1/s}} \le C\left\|\left\{\sum_{j=1}^{\infty}\left|f_{j}\right|^{r}\right\}^{1/r}\right\|_{\left[K_{p,q}^{\alpha}\left(\mathbb{R}^{n}\right)\right]^{1/s}}.
	\end{align*}
\end{lemma}

Furthermore, the following lemma shows that the Fefferman--Stein vector-valued maximal inequalities for $WK_{p,q}^{\alpha}\left(\mathbb{R}^{n}\right)$ hold true.

\begin{lemma}\label{le:Herz-as-(WX)^1/s}
	\textup{(see \citep[Lemma 6.1.6]{LiYangHuang} or \citep[Threorem 4.4]{SunYangYuan}) Let $p\in\left(1,\infty\right)$, $q\in\left(0,\infty\right)$ and $\alpha\in\left(-n/p,\infty\right)$. Assume that $r\in\left(1,\infty\right)$ and $s\in\left(0,\min{\left\{p,\left[\alpha/n+1/p\right]^{-1}\right\}}\right)$. Then there exists a positive constant $C$ such that for any $\left\{f_{j}\right\}_{j=1}^{\infty}\subset \mathscr{M}\left(\mathbb{R}^{n}\right)$,}
	\begin{align*}
		\left\|\left\{\sum_{j=1}^{\infty}\left[M\left(f_{j}\right)\right]^{r}\right\}^{1/r}\right\|_{\left[WK_{p,q}^{\alpha}\left(\mathbb{R}^{n}\right)\right]^{1/s}} \le C\left\|\left\{\sum_{j=1}^{\infty}\left|f_{j}\right|^{r}\right\}^{1/r}\right\|_{\left[WK_{p,q}^{\alpha}\left(\mathbb{R}^{n}\right)\right]^{1/s}}.
	\end{align*}
\end{lemma}

The following lemma shows that the Herz space $K_{p,q}^{\alpha}\left(\mathbb{R}^{n}\right)$ satisfies Assumptions \ref{as:(X^{1/s})'}.

\begin{lemma}\label{le:Herz-as-(X^{1/r})'}
	\textup{\citep[Remark 2.7(c)]{WangYangYang} Let $p,q\in\left(0,\infty\right)$ and $\alpha\in\left(-n/p,\infty\right)$. Then, for any $r\in\left(0,\min{\left\{p,q\right\}}\right)$ and $s\in\left(\max{\left\{1,p\right\}},\infty\right]$, there exists a positive constant $C$ such that for any $\left\{f\right\}_{j=1}^{\infty}\subset \mathscr{M}\left(\mathbb{R}^{n}\right)$,}
	\begin{align*}
		\left\|M^{\left((s/r)'\right)}(f)\right\|_{\left(\left[K_{p,q}^{\alpha}\left(\mathbb{R}^{n}\right)\right]^{1/r}\right)'} \le C\|f\|_{\left(\left[K_{p,q}^{\alpha}\left(\mathbb{R}^{n}\right)\right]^{1/r}\right)'},
	\end{align*}
	\textup{where $\left(\left[K_{p,q}^{\alpha}\left(\mathbb{R}^{n}\right)\right]^{1/r}\right)'$ is as in (\ref{eq:X'}) with $X:=\left[K_{p,q}^{\alpha}\left(\mathbb{R}^{n}\right)\right]^{1/r}$.}
\end{lemma}

Applying Lemmas \ref{le:Herz-as-X^1/s}, \ref{le:Herz-as-(WX)^1/s} and \ref{le:Herz-as-(X^{1/r})'}, Proposition \ref{pro:X}, Theorems \ref{th:B-R Hx-WHx}, \ref{th:B-R Hx-WX} and \ref{th:MB-R Hx-WX}, we immediately obtain the following boundedness of Bochner-Riesz means and the maximal Bochner--Riesz means, respectively, as follows.

\begin{theorem}\label{th:Herz spaces-B-R Hx-WHx}
	\textup{Let $p\in\left(0,1\right]$, $q\in\left(0,\infty\right)$, $\alpha\in\left(-n/p,\infty\right)$ and $\delta>\frac{n-1}{2}$. Let $B_{1/\varepsilon}^{\delta}$ be a Bochner-Riesz means with $\delta$ order on $\mathbb{R}^{n}$. If $\min{\left\{p,\left[\alpha/n+1/p\right]^{-1}\right\}}\in\left[\frac{2n}{n+1+2\delta},1\right]$, then $B_{1/\varepsilon}^{\delta}$ has a unique extension on $HK_{p,q}^{\alpha}\left(\mathbb{R}^{n}\right)$. Moreover, there exists a positive constant $C$ such that for any $f\in HK_{p,q}^{\alpha}\left(\mathbb{R}^{n}\right)$,}
	\begin{align*}
		\left\|B_{1/\varepsilon}^{\delta}\left(f\right)\right\|_{WHK_{p,q}^{\alpha}\left(\mathbb{R}^{n}\right)} \le C\left\|f\right\|_{HK_{p,q}^{\alpha}\left(\mathbb{R}^{n}\right)}.
	\end{align*}
\end{theorem}

\begin{theorem}\label{th:Herz spaces-B-R Hx-WX}
	\textup{Let $p\in\left(0,1\right]$, $q\in\left(0,\infty\right)$, $\alpha\in\left(-n/p,\infty\right)$ and $\delta>\frac{n-1}{2}$. Let $B_{1/\varepsilon}^{\delta}$ be a Bochner-Riesz means with $\delta$ order on $\mathbb{R}^{n}$. If $\min{\left\{p,\left[\alpha/n+1/p\right]^{-1}\right\}}\in\left[\frac{2n}{n+1+2\delta},1\right]$, then $B_{1/\varepsilon}^{\delta}$ has a unique extension on $HK_{p,q}^{\alpha}\left(\mathbb{R}^{n}\right)$. Moreover, there exists a positive constant $C$ such that for any $f\in HK_{p,q}^{\alpha}\left(\mathbb{R}^{n}\right)$,}
	\begin{align*}
		\left\|B_{1/\varepsilon}^{\delta}\left(f\right)\right\|_{WK_{p,q}^{\alpha}\left(\mathbb{R}^{n}\right)} \le C\left\|f\right\|_{HK_{p,q}^{\alpha}\left(\mathbb{R}^{n}\right)}.
	\end{align*}
\end{theorem}

\begin{theorem}\label{th:Herz spaces-MB-R Hx-WX}
	\textup{Let $p\in\left(0,1\right]$, $q\in\left(0,\infty\right)$, $\alpha\in\left(-n/p,\infty\right)$ and $\delta>\frac{n-1}{2}$. Let $B_{\ast}^{\delta}$ be the maximal Bochner--Riesz means with $\delta$ order on $\mathbb{R}^{n}$. If $\min{\left\{p,\left[\alpha/n+1/p\right]^{-1}\right\}}\in\left[\frac{2n}{n+1+2\delta},1\right]$, then $B_{\ast}^{\delta}$ has a unique extension on $HK_{p,q}^{\alpha}\left(\mathbb{R}^{n}\right)$. Moreover, there exists a positive constant $C$ such that for any $f\in HK_{p,q}^{\alpha}\left(\mathbb{R}^{n}\right)$,}
	\begin{align*}
		\left\|B_{\ast}^{\delta}\left(f\right)\right\|_{WK_{p,q}^{\alpha}\left(\mathbb{R}^{n}\right)} \le C\left\|f\right\|_{HK_{p,q}^{\alpha}\left(\mathbb{R}^{n}\right)}.
	\end{align*}
\end{theorem}

\subsection{Lorentz spaces.}
We first recall the notion of Lorentz spaces.

\begin{definition}\label{de:Lorentz spaces}
	\textup{The Lorentz space $L^{p,q}\left(\mathbb{R}^{n}\right)$ is defined to be the set of all measurable functions $f$ on $\mathbb{R}^{n}$ satisfying that, when $p,q\in\left(0,\infty\right)$,}
	\begin{align*}
		\left\|f\right\|_{L^{p,q}\left(\mathbb{R}^{n}\right)}:=\left\{\int_{0}^{\infty} \left[t^{1/p}f^{\ast}\left(t\right)\right]^{q}\frac{\,dt}{t} \right\}^{\frac{1}{q}} <\infty,
	\end{align*}
	\textup{and, when $p\in\left(0,\infty\right)$ and $q=\infty$,}
	\begin{align*}
		\left\|f\right\|_{L^{p,q}\left(\mathbb{R}^{n}\right)}:=\sup\limits_{t\in\left(0,\infty\right)} \left\{t^{1/p}f^{\ast}\left(t\right)\right\}<\infty,
	\end{align*}
    \textup{where $f^{\ast}$, the decreasing rearrangement function of $f$, is defined by setting, for any $t\in\left[0,\infty\right)$,}
    \begin{align*}
    	f^{\ast}\left(t\right):=\inf\left\{s\in\left(0,\infty\right):\mu_{f}\left(s\right)\le t\right\}
    \end{align*}
    \textup{with $\mu_{f}\left(s\right):=\left|\left\{x\in\mathbb{R}^{n}:\left|f\left(x\right)\right|>s\right\}\right|$.}
\end{definition}

Let $X$ be a ball quasi-Banach function space. Let $p,q\in\left(0,\infty\right)$. If $X:=L^{p,q}\left(\mathbb{R}^{n}\right)$, then $WX:=WL^{p,q}\left(\mathbb{R}^{n}\right)$ is the weak Lorentz space, $H_{X}\left(\mathbb{R}^{n}\right):=H^{p,q}\left(\mathbb{R}^{n}\right)$ is the Hardy--Lorentz space (when $p,q\in\left(0,\infty\right)$), $H_{X}\left(\mathbb{R}^{n}\right):=H^{p,q}\left(\mathbb{R}^{n}\right)$ is the weak Hardy space (when $p\in\left(0,\infty\right)$ and $q=\infty$), and $WH_{X}\left(\mathbb{R}^{n}\right):=WH^{p,q}\left(\mathbb{R}^{n}\right)$ is the weak Hardy-Lorentz space. The Definitions of weak Lorentz spaces, Hardy--Lorentz spaces and weak Hardy-Lorentz spaces are as in Definition \ref{de:WX}, \ref{de:the Hardy type space} and \ref{de:the weak Hardy-type space} with $X$ replaced by $L^{p,q}\left(\mathbb{R}^{n}\right)$.

Moreover, the Lorentz space $L^{p,q}\left(\mathbb{R}^{n}\right)$ satisfies Assumptions \ref{as:the boundedness of M on X^1/p}, \ref{as:M is bounded on (WX)^{1/r}} and \ref{as:(X^{1/s})'}, respectively, in Lemmas \ref{le:Lorentz-as-X^1/s}, \ref{le:Lorentz-as-(WX)^1/s} and \ref{le:Lorentz-as-(X^{1/r})'}. The following lemma shows that the Fefferman-Stein vector-valued maximal inequalities for $L^{p,q}\left(\mathbb{R}^{n}\right)$ hold true.

\begin{lemma}\label{le:Lorentz-as-X^1/s}
	\textup{(see \citep[Lemma 4.5]{LiuYangYuan} or \citep[Theorem 3.4]{JiaoZuoZhouWu}) Let $p\in\left(1,\infty\right)$ and $q\in\left(0,\infty\right]$. Assume that $r\in\left(1,\infty\right)$ and $s\in\left(0,\min{\left\{p,q\right\}}\right)$. Then there exists a positive constant $C$ such that for any $\left\{f_{j}\right\}_{j=1}^{\infty}\subset \mathscr{M}\left(\mathbb{R}^{n}\right)$,}
	\begin{align*}
		\left\|\left\{\sum_{j=1}^{\infty}\left[M\left(f_{j}\right)\right]^{r}\right\}^{1/r}\right\|_{\left[L^{p,q}\left(\mathbb{R}^{n}\right)\right]^{1/s}} \le C\left\|\left\{\sum_{j=1}^{\infty}\left|f_{j}\right|^{r}\right\}^{1/r}\right\|_{\left[L^{p,q}\left(\mathbb{R}^{n}\right)\right]^{1/s}}.
	\end{align*}
\end{lemma}

Furthermore, the following lemma shows that the Fefferman--Stein vector-valued maximal inequalities for $WL^{p,q}\left(\mathbb{R}^{n}\right)$ hold true.

\begin{lemma}\label{le:Lorentz-as-(WX)^1/s}
	\textup{(see \citep[Lemma 6.1.6]{LiYangHuang} or \citep[Threorem 4.4]{SunYangYuan}) Let $p\in\left(1,\infty\right)$ and $q\in\left(0,\infty\right]$. Assume that $r\in\left(1,\infty\right)$ and $s\in\left(0,\min{\left\{p,q\right\}}\right)$. Then there exists a positive constant $C$ such that for any $\left\{f_{j}\right\}_{j=1}^{\infty}\subset \mathscr{M}\left(\mathbb{R}^{n}\right)$,}
	\begin{align*}
		\left\|\left\{\sum_{j=1}^{\infty}\left[M\left(f_{j}\right)\right]^{r}\right\}^{1/r}\right\|_{\left[WL^{p,q}\left(\mathbb{R}^{n}\right)\right]^{1/s}} \le C\left\|\left\{\sum_{j=1}^{\infty}\left|f_{j}\right|^{r}\right\}^{1/r}\right\|_{\left[WL^{p,q}\left(\mathbb{R}^{n}\right)\right]^{1/s}}.
	\end{align*}
\end{lemma}

The following lemma shows that the Lorentz space $L^{p,q}\left(\mathbb{R}^{n}\right)$ satisfies Assumptions \ref{as:(X^{1/s})'}.

\begin{lemma}\label{le:Lorentz-as-(X^{1/r})'}
	\textup{\citep[Remark 2.7(d)]{WangYangYang} Let $p,q\in\left(0,\infty\right)$. Then, for any $r\in\left(0,\min{\left\{1,p,q\right\}}\right)$ and $s\in\left(\max{\left\{1,p,q\right\}},\infty\right]$, there exists a positive constant $C$ such that for any $\left\{f\right\}_{j=1}^{\infty}\subset \mathscr{M}\left(\mathbb{R}^{n}\right)$,}
	\begin{align*}
		\left\|M^{\left((s/r)'\right)}(f)\right\|_{\left(\left[L^{p,q}\left(\mathbb{R}^{n}\right)\right]^{1/r}\right)'} \le C\|f\|_{\left(\left[L^{p,q}\left(\mathbb{R}^{n}\right)\right]^{1/r}\right)'},
	\end{align*}
	\textup{where $\left(\left[L^{p,q}\left(\mathbb{R}^{n}\right)\right]^{1/r}\right)'$ is as in (\ref{eq:X'}) with $X:=\left[L^{p,q}\left(\mathbb{R}^{n}\right)\right]^{1/r}$.}
\end{lemma}

Applying Lemmas \ref{le:Lorentz-as-X^1/s}, \ref{le:Lorentz-as-(WX)^1/s} and \ref{le:Lorentz-as-(X^{1/r})'}, Proposition \ref{pro:X}, Theorems \ref{th:B-R Hx-WHx}, \ref{th:B-R Hx-WX} and \ref{th:MB-R Hx-WX}, we immediately obtain the following boundedness of Bochner--Riesz means and the maximal Bochner--Riesz means, respectively, as follows.

\begin{theorem}\label{th:Lorentz spaces-B-R Hx-WHx}
	\textup{Let $p\in\left(0,1\right]$, $q\in\left(0,\infty\right)$ and $\delta>\frac{n-1}{2}$. Let $B_{1/\varepsilon}^{\delta}$ be a Bochner-Riesz means with $\delta$ order on $\mathbb{R}^{n}$. If $\min{\left\{p,q\right\}}\in\left[\frac{2n}{n+1+2\delta},1\right]$, then $B_{1/\varepsilon}^{\delta}$ has a unique extension on $H^{p,q}\left(\mathbb{R}^{n}\right)$. Moreover, there exists a positive constant $C$ such that for any $f\in H^{p,q}\left(\mathbb{R}^{n}\right)$,}
	\begin{align*}
		\left\|B_{1/\varepsilon}^{\delta}\left(f\right)\right\|_{WH^{p,q}\left(\mathbb{R}^{n}\right)} \le C\left\|f\right\|_{H^{p,q}\left(\mathbb{R}^{n}\right)}.
	\end{align*}
\end{theorem}

\begin{theorem}\label{th:Lorentz spaces-B-R Hx-WX}
	\textup{Let $p\in\left(0,1\right]$, $q\in\left(0,\infty\right)$ and $\delta>\frac{n-1}{2}$. Let $B_{1/\varepsilon}^{\delta}$ be a Bochner--Riesz means with $\delta$ order on $\mathbb{R}^{n}$. If $\min{\left\{p,q\right\}}\in\left[\frac{2n}{n+1+2\delta},1\right]$, then $B_{1/\varepsilon}^{\delta}$ has a unique extension on $H^{p,q}\left(\mathbb{R}^{n}\right)$. Moreover, there exists a positive constant $C$ such that for any $f\in H^{p,q}\left(\mathbb{R}^{n}\right)$,}
	\begin{align*}
		\left\|B_{1/\varepsilon}^{\delta}\left(f\right)\right\|_{WL^{p,q}\left(\mathbb{R}^{n}\right)} \le C\left\|f\right\|_{H^{p,q}\left(\mathbb{R}^{n}\right)}.
	\end{align*}
\end{theorem}

\begin{theorem}\label{th:Lorentz spaces-MB-R Hx-WX}
	\textup{Let $p\in\left(0,1\right]$, $q\in\left(0,\infty\right)$ and $\delta>\frac{n-1}{2}$. Let $B_{\ast}^{\delta}$ be the maximal Bochner--Riesz means with $\delta$ order on $\mathbb{R}^{n}$. If $\min{\left\{p,q\right\}}\in\left[\frac{2n}{n+1+2\delta},1\right]$, then $B_{\ast}^{\delta}$ has a unique extension on $H^{p,q}\left(\mathbb{R}^{n}\right)$. Moreover, there exists a positive constant $C$ such that for any $f\in H^{p,q}\left(\mathbb{R}^{n}\right)$,}
	\begin{align*}
		\left\|B_{\ast}^{\delta}\left(f\right)\right\|_{WL^{p,q}\left(\mathbb{R}^{n}\right)} \le C\left\|f\right\|_{H^{p,q}\left(\mathbb{R}^{n}\right)}.
	\end{align*}
\end{theorem}

\subsection{Variable Lebesgue spaces.}
We begin with the notion of Variable Lebesgue spaces.

\begin{definition}\label{de:Variable Lebesgue spaces}
	\textup{Let $p\left(\cdot\right): \mathbb{R}^{n}\to \left[0,\infty\right)$ be a measurable function. Then the variable Lebesgue space $L^{p\left(\cdot\right)}\left(\mathbb{R}^{n}\right)$ is defined to be the set of all measurable functions $f$ on $\mathbb{R}^{n}$ such that}
	\begin{align*}
		\left\|f\right\|_{L^{p\left(\cdot\right)}\left(\mathbb{R}^{n}\right)}:=\inf\left\{\lambda\in\left(0,\infty\right):\int_{\mathbb{R}^{n}} \left[\frac{\left|f\left(x\right)\right|}{\lambda}\right]^{p\left(x\right)}\,dx\le 1 \right\} <\infty.
	\end{align*}
\end{definition}

Let $X$ be a ball quasi-Banach function space. Let $p\left(\cdot\right): \mathbb{R}^{n}\to \left[0,\infty\right)$ be a measurable function. If $X:=L^{p\left(\cdot\right)}\left(\mathbb{R}^{n}\right)$, then $WX:=WL^{p\left(\cdot\right)}\left(\mathbb{R}^{n}\right)$ is the weak variable Lebesgue space, $H_{X}\left(\mathbb{R}^{n}\right):=H^{p\left(\cdot\right)}\left(\mathbb{R}^{n}\right)$ is the variable Hardy space, and $WH_{X}\left(\mathbb{R}^{n}\right):=WH^{p\left(\cdot\right)}\left(\mathbb{R}^{n}\right)$ is the weak variable Hardy space. The Definitions of weak variable Lebesgue spaces, variable Hardy spaces and weak variable Hardy spaces are as in Definition \ref{de:WX}, \ref{de:the Hardy type space} and \ref{de:the weak Hardy-type space} with $X$ replaced by $L^{p\left(\cdot\right)}\left(\mathbb{R}^{n}\right)$.

Moreover, the variable Lebesgue space $L^{p\left(\cdot\right)}\left(\mathbb{R}^{n}\right)$ satisfies Assumptions \ref{as:the boundedness of M on X^1/p}, \ref{as:M is bounded on (WX)^{1/r}} and \ref{as:(X^{1/s})'}, respectively, in Lemmas \ref{le:Variable Lebesgue-as-X^1/s}, \ref{le:Variable Lebesgue-as-(WX)^1/s} and \ref{le:Variable Lebesgue-as-(X^{1/r})'}. The following lemma shows that the Fefferman--Stein vector-valued maximal inequalities for $L^{p\left(\cdot\right)}\left(\mathbb{R}^{n}\right)$ hold true.

\begin{lemma}\label{le:Variable Lebesgue-as-X^1/s}
	\textup{\citep[Corollary 4.3]{Ho2016} Let $p\left(\cdot\right)\in LH$ with $1<p_{-}\le p_{+}<\infty$. Assume that $r\in\left(1,\infty\right)$ and $s\in\left(0,p_{-}\right)$. Then there exists a positive constant $C$ such that for any $\left\{f_{j}\right\}_{j=1}^{\infty}\subset \mathscr{M}\left(\mathbb{R}^{n}\right)$,}
	\begin{align*}
		\left\|\left\{\sum_{j=1}^{\infty}\left[M\left(f_{j}\right)\right]^{r}\right\}^{1/r}\right\|_{\left[L^{p\left(\cdot\right)}\left(\mathbb{R}^{n}\right)\right]^{1/s}} \le C\left\|\left\{\sum_{j=1}^{\infty}\left|f_{j}\right|^{r}\right\}^{1/r}\right\|_{\left[L^{p\left(\cdot\right)}\left(\mathbb{R}^{n}\right)\right]^{1/s}},
	\end{align*}
    \textup{where $p\left(\cdot\right)\in LH$ denotes that $p\left(\cdot\right)$ satisfies the log-H{\"o}lder continuity condition and the log-H{\"o}lder continuity condition at infinity.}
\end{lemma}

Furthermore, the following lemma shows that the Fefferman--Stein vector-valued maximal inequalities for $WL^{p\left(\cdot\right)}\left(\mathbb{R}^{n}\right)$ hold true.

\begin{lemma}\label{le:Variable Lebesgue-as-(WX)^1/s}
	\textup{(see \citep[Lemma 6.1.6]{LiYangHuang} or \citep[Threorem 4.4]{SunYangYuan}) Let $p\left(\cdot\right)\in LH$ with $1<p_{-}\le p_{+}<\infty$. Assume that $r\in\left(1,\infty\right)$ and $s\in\left(0,p_{-}\right)$. Then there exists a positive constant $C$ such that for any $\left\{f_{j}\right\}_{j=1}^{\infty}\subset \mathscr{M}\left(\mathbb{R}^{n}\right)$,}
	\begin{align*}
		\left\|\left\{\sum_{j=1}^{\infty}\left[M\left(f_{j}\right)\right]^{r}\right\}^{1/r}\right\|_{\left[WL^{p\left(\cdot\right)}\left(\mathbb{R}^{n}\right)\right]^{1/s}} \le C\left\|\left\{\sum_{j=1}^{\infty}\left|f_{j}\right|^{r}\right\}^{1/r}\right\|_{\left[WL^{p\left(\cdot\right)}\left(\mathbb{R}^{n}\right)\right]^{1/s}}.
	\end{align*}
\end{lemma}

The following lemma shows that the variable Lebesgue space $L^{p\left(\cdot\right)}\left(\mathbb{R}^{n}\right)$ satisfies Assumptions \ref{as:(X^{1/s})'}.

\begin{lemma}\label{le:Variable Lebesgue-as-(X^{1/r})'}
	\textup{\citep[Remark 2.7(f)]{WangYangYang} Let $p\left(\cdot\right)\in LH$ with $0<p_{-}\le p_{+}<\infty$. Then, for any $r\in\left(0,\min{\left\{1,p_{-}\right\}}\right)$ and $s\in\left(\max{\left\{1,p_{+}\right\}},\infty\right]$, there exists a positive constant $C$ such that for any $\left\{f\right\}_{j=1}^{\infty}\subset \mathscr{M}\left(\mathbb{R}^{n}\right)$,}
	\begin{align*}
		\left\|M^{\left((s/r)'\right)}(f)\right\|_{\left(\left[L^{p\left(\cdot\right)}\left(\mathbb{R}^{n}\right)\right]^{1/r}\right)'} \le C\|f\|_{\left(\left[L^{p\left(\cdot\right)}\left(\mathbb{R}^{n}\right)\right]^{1/r}\right)'},
	\end{align*}
	\textup{where $\left(\left[L^{p\left(\cdot\right)}\left(\mathbb{R}^{n}\right)\right]^{1/r}\right)'$ is as in (\ref{eq:X'}) with $X:=\left[L^{p\left(\cdot\right)}\left(\mathbb{R}^{n}\right)\right]^{1/r}$.}
\end{lemma}

Applying Lemmas \ref{le:Variable Lebesgue-as-X^1/s}, \ref{le:Variable Lebesgue-as-(WX)^1/s} and \ref{le:Variable Lebesgue-as-(X^{1/r})'}, Proposition \ref{pro:X}, Theorems \ref{th:B-R Hx-WHx}, \ref{th:B-R Hx-WX} and \ref{th:MB-R Hx-WX}, we immediately obtain the following boundedness of Bochner-Riesz means and the maximal Bochner-Riesz means, respectively, as follows.

\begin{theorem}\label{th:Variable Lebesgue spaces-B-R Hx-WHx}
	\textup{Let $\delta>\frac{n-1}{2}$ and $p\left(\cdot\right): \mathbb{R}^{n}\to \left(0,1\right]$ be globally log-H{\"o}lder continuous with $p_{-}$ and $p_{+}$. Let $B_{1/\varepsilon}^{\delta}$ be a Bochner-Riesz means with $\delta$ order on $\mathbb{R}^{n}$. If $p_{-}\in\left[\frac{2n}{n+1+2\delta},1\right]$, then $B_{1/\varepsilon}^{\delta}$ has a unique extension on $H^{p\left(\cdot\right)}\left(\mathbb{R}^{n}\right)$. Moreover, there exists a positive constant $C$ such that for any $f\in H^{p\left(\cdot\right)}\left(\mathbb{R}^{n}\right)$,}
	\begin{align*}
		\left\|B_{1/\varepsilon}^{\delta}\left(f\right)\right\|_{WH^{p\left(\cdot\right)}\left(\mathbb{R}^{n}\right)} \le C\left\|f\right\|_{H^{p\left(\cdot\right)}\left(\mathbb{R}^{n}\right)}.
	\end{align*}
\end{theorem}

\begin{theorem}\label{th:Variable Lebesgue spaces-B-R Hx-WX}
	\textup{Let $\delta>\frac{n-1}{2}$ and $p\left(\cdot\right): \mathbb{R}^{n}\to \left(0,1\right]$ be globally log-H{\"o}lder continuous with $p_{-}$ and $p_{+}$. Let $B_{1/\varepsilon}^{\delta}$ be a Bochner--Riesz means with $\delta$ order on $\mathbb{R}^{n}$. If $p_{-}\in\left[\frac{2n}{n+1+2\delta},1\right]$, then $B_{1/\varepsilon}^{\delta}$ has a unique extension on $H^{p\left(\cdot\right)}\left(\mathbb{R}^{n}\right)$. Moreover, there exists a positive constant $C$ such that for any $f\in H^{p\left(\cdot\right)}\left(\mathbb{R}^{n}\right)$,}
	\begin{align*}
		\left\|B_{1/\varepsilon}^{\delta}\left(f\right)\right\|_{WL^{p\left(\cdot\right)}\left(\mathbb{R}^{n}\right)} \le C\left\|f\right\|_{H^{p\left(\cdot\right)}\left(\mathbb{R}^{n}\right)}.
	\end{align*}
\end{theorem}

\begin{theorem}\label{th:Variable Lebesgue spaces-MB-R Hx-WX}
	\textup{Let $\delta>\frac{n-1}{2}$ and $p\left(\cdot\right): \mathbb{R}^{n}\to \left(0,1\right]$ be globally log-H{\"o}lder continuous with $p_{-}$ and $p_{+}$. Let $B_{\ast}^{\delta}$ be the maximal Bochner--Riesz means with $\delta$ order on $\mathbb{R}^{n}$. If $p_{-}\in\left[\frac{2n}{n+1+2\delta},1\right]$, then $B_{\ast}^{\delta}$ has a unique extension on $H^{p\left(\cdot\right)}\left(\mathbb{R}^{n}\right)$. Moreover, there exists a positive constant $C$ such that for any $f\in H^{p\left(\cdot\right)}\left(\mathbb{R}^{n}\right)$,}
	\begin{align*}
		\left\|B_{\ast}^{\delta}\left(f\right)\right\|_{WL^{p\left(\cdot\right)}\left(\mathbb{R}^{n}\right)} \le C\left\|f\right\|_{H^{p\left(\cdot\right)}\left(\mathbb{R}^{n}\right)}.
	\end{align*}
\end{theorem}

\subsection{Morrey spaces.}
We first recall the notion of Morrey spaces.
\begin{definition}\label{de:Morrey spaces}
	\textup{Let $0<q\le p\le \infty$. The Morrey space $\mathcal{M}_{q}^{p}\left(\mathbb{R}^{n}\right)$ is defined to be the set of all $f\in L_{\rm loc}^{q}\left(\mathbb{R}^{n}\right)$ such that}
	\begin{align*}
		\left\|f\right\|_{\mathcal{M}_{q}^{p}\left(\mathbb{R}^{n}\right)}:=\sup\limits_{B\in\mathbb{B}} \left|B\right|^{1/p-1/q}\left\{\int_{B}\left|f\left(y\right)\right|^{q}\,dy\right\}^{1/q}<\infty,
	\end{align*}
	\textup{where $\mathbb{B}$ is as in (\ref{eq:B(x,r)}).}
\end{definition}

Let $X$ be a ball quasi-Banach function space. Let $0<q\le p\le \infty$. If $X:=\mathcal{M}_{q}^{p}\left(\mathbb{R}^{n}\right)$, then $WX:=W\mathcal{M}_{q}^{p}\left(\mathbb{R}^{n}\right)$ is the weak Morrey space, $H_{X}\left(\mathbb{R}^{n}\right):=H\mathcal{M}_{q}^{p}\left(\mathbb{R}^{n}\right)$ is the Hardy--Morrey space, and $WH_{X}\left(\mathbb{R}^{n}\right):=WH\mathcal{M}_{q}^{p}\left(\mathbb{R}^{n}\right)$ is the weak Hardy--Morrey space. The Definitions of weak Morrey spaces, Hardy--Morrey spaces and weak Hardy--Morrey spaces are as in Definition \ref{de:WX}, \ref{de:the Hardy type space} and \ref{de:the weak Hardy-type space} with $X$ replaced by $\mathcal{M}_{q}^{p}\left(\mathbb{R}^{n}\right)$.

The space $\mathcal{M}_{q}^{p}\left(\mathbb{R}^{n}\right)$ was introduced by Morrey \cite{Morrey} in 1938. Moreover, the Morrey space $\mathcal{M}_{q}^{p}\left(\mathbb{R}^{n}\right)$ satisfies Assumptions \ref{as:the boundedness of M on X^1/p}, \ref{as:M is bounded on (WX)^{1/r}} and \ref{as:(X^{1/s})'}, respectively, in Lemmas \ref{le:Morrey-as-X^1/s}, \ref{le:Morrey-as-(WX)^1/s} and \ref{le:Morrey-as-(X^{1/r})'}. The following lemma shows that the Fefferman--Stein vector-valued maximal inequalities for $\mathcal{M}_{q}^{p}\left(\mathbb{R}^{n}\right)$ hold true.

\begin{lemma}\label{le:Morrey-as-X^1/s}
	\textup{\citep[Lemma 2.5]{TangXu} Let $0<q\le p<\infty$. Assume that $r\in\left(1,\infty\right)$ and $s\in\left(0,q\right)$. Then there exists a positive constant $C$ such that for any $\left\{f_{j}\right\}_{j=1}^{\infty}\subset \mathscr{M}\left(\mathbb{R}^{n}\right)$,}
	\begin{align*}
		\left\|\left\{\sum_{j=1}^{\infty}\left[M\left(f_{j}\right)\right]^{r}\right\}^{1/r}\right\|_{\left[\mathcal{M}_{q}^{p}\left(\mathbb{R}^{n}\right)\right]^{1/s}} \le C\left\|\left\{\sum_{j=1}^{\infty}\left|f_{j}\right|^{r}\right\}^{1/r}\right\|_{\left[\mathcal{M}_{q}^{p}\left(\mathbb{R}^{n}\right)\right]^{1/s}}.
	\end{align*}
\end{lemma}

Furthermore, the following lemma shows that the Fefferman--Stein vector-valued maximal inequalities for $W\mathcal{M}_{q}^{p}\left(\mathbb{R}^{n}\right)$ hold true.

\begin{lemma}\label{le:Morrey-as-(WX)^1/s}
	\textup{\citep[Theorem 3.2]{Ho2017} Let $0<q\le p<\infty$. Assume that $r\in\left(1,\infty\right)$ and $s\in\left(0,q\right)$. Then there exists a positive constant $C$ such that for any $\left\{f_{j}\right\}_{j=1}^{\infty}\subset \mathscr{M}\left(\mathbb{R}^{n}\right)$,}
	\begin{align*}
		\left\|\left\{\sum_{j=1}^{\infty}\left[M\left(f_{j}\right)\right]^{r}\right\}^{1/r}\right\|_{\left[W\mathcal{M}_{q}^{p}\left(\mathbb{R}^{n}\right)\right]^{1/s}} \le C\left\|\left\{\sum_{j=1}^{\infty}\left|f_{j}\right|^{r}\right\}^{1/r}\right\|_{\left[W\mathcal{M}_{q}^{p}\left(\mathbb{R}^{n}\right)\right]^{1/s}}.
	\end{align*}
\end{lemma}

The following lemma shows that the Morrey space $\mathcal{M}_{q}^{p}\left(\mathbb{R}^{n}\right)$ satisfies Assumptions \ref{as:(X^{1/s})'}.

\begin{lemma}\label{le:Morrey-as-(X^{1/r})'}
	\textup{\citep[Lemma 7.6]{ZhangYangYuan} Let $0<q\le p<\infty$, $r\in\left(0,q\right)$ and $s\in\left(q,\infty\right]$. Then there exists a positive constant $C$ such that for any $\left\{f\right\}_{j=1}^{\infty}\subset \mathscr{M}\left(\mathbb{R}^{n}\right)$,}
	\begin{align*}
		\left\|M^{\left((s/r)'\right)}(f)\right\|_{\left(\left[\mathcal{M}_{q}^{p}\left(\mathbb{R}^{n}\right)\right]^{1/r}\right)'} \le C\|f\|_{\left(\left[\mathcal{M}_{q}^{p}\left(\mathbb{R}^{n}\right)\right]^{1/r}\right)'},
	\end{align*}
    \textup{where $\left(\left[\mathcal{M}_{q}^{p}\left(\mathbb{R}^{n}\right)\right]^{1/r}\right)'$ is as in (\ref{eq:X'}) with $X:=\left[\mathcal{M}_{q}^{p}\left(\mathbb{R}^{n}\right)\right]^{1/r}$.}
\end{lemma}

To apply Theorems \ref{th:B-R Hx-WHx}, \ref{th:B-R Hx-WX} and \ref{th:MB-R Hx-WX} to Morrey spaces, we need the following weak-type Fefferman--Stein vector-valued inequality of the Hardy-Littlewood maximal operator $M$ in (\ref{eq:the Hardy-Littlewood maximal operator}) from $\mathcal{M}_{1}^{p}\left(\mathbb{R}^{n}\right)$ to $W\mathcal{M}_{1}^{p}\left(\mathbb{R}^{n}\right)$ \citep[Proposition 7.16]{ZhangYangYuan}.

\begin{proposition}\label{pro:Morrey spaces}
	\textup{Let $p\in\left[1,\infty\right)$ and $r\in\left(1,\infty\right)$. Then there exists a positive constant $C$ such that for any $\left\{f_{j}\right\}_{j\in\mathbb{N}}\subset \mathcal{M}_{1}^{p}\left(\mathbb{R}^{n}\right)$,}
	\begin{align*}
		\left\|\left\{\sum_{j=1}^{\infty}\left[M\left(f_{j}\right)\right]^{r}\right\}^{\frac{1}{r}}\right\|_{W\mathcal{M}_{1}^{p}\left(\mathbb{R}^{n}\right)}\le C\left\|\left\{\sum_{j=1}^{\infty}\left|f_{j}\right|^{r}\right\}^{\frac{1}{r}}\right\|_{\mathcal{M}_{1}^{p}\left(\mathbb{R}^{n}\right)}.
	\end{align*}
\end{proposition}

Applying Lemmas \ref{le:Morrey-as-X^1/s}, \ref{le:Morrey-as-(WX)^1/s} and \ref{le:Morrey-as-(X^{1/r})'}, Proposition \ref{pro:Morrey spaces}, Theorems \ref{th:B-R Hx-WHx}, \ref{th:B-R Hx-WX} and \ref{th:MB-R Hx-WX}, we immediately obtain the following boundedness of Bochner-Riesz means and the maximal Bochner-Riesz means, respectively, as follows.

\begin{theorem}\label{th:Morrey spaces-B-R Hx-WHx}
	\textup{Let $q\in\left(0,1\right]$, $p\in\left(0,\infty\right)$ with $q\le p$, and $\delta>\frac{n-1}{2}$. Let $B_{1/\varepsilon}^{\delta}$ be a Bochner-Riesz means with $\delta$ order on $\mathbb{R}^{n}$. If $q\in\left[\frac{2n}{n+1+2\delta},1\right]$, then $B_{1/\varepsilon}^{\delta}$ has a unique extension on $H\mathcal{M}_{q}^{p}\left(\mathbb{R}^{n}\right)$. Moreover, there exists a positive constant $C$ such that for any $f\in H\mathcal{M}_{q}^{p}\left(\mathbb{R}^{n}\right)$,}
	\begin{align*}
		\left\|B_{1/\varepsilon}^{\delta}\left(f\right)\right\|_{WH\mathcal{M}_{q}^{p}\left(\mathbb{R}^{n}\right)} \le C\left\|f\right\|_{H\mathcal{M}_{q}^{p}\left(\mathbb{R}^{n}\right)}.
	\end{align*}
\end{theorem}

\begin{theorem}\label{th:Morrey spaces-B-R Hx-WX}
	\textup{Let $q\in\left(0,1\right]$, $p\in\left(0,\infty\right)$ with $q\le p$, and $\delta>\frac{n-1}{2}$. Let $B_{1/\varepsilon}^{\delta}$ be a Bochner-Riesz means with $\delta$ order on $\mathbb{R}^{n}$. If $q\in\left[\frac{2n}{n+1+2\delta},1\right]$, then $B_{1/\varepsilon}^{\delta}$ has a unique extension on $H\mathcal{M}_{q}^{p}\left(\mathbb{R}^{n}\right)$. Moreover, there exists a positive constant $C$ such that for any $f\in H\mathcal{M}_{q}^{p}\left(\mathbb{R}^{n}\right)$,}
	\begin{align*}
		\left\|B_{1/\varepsilon}^{\delta}\left(f\right)\right\|_{W\mathcal{M}_{q}^{p}\left(\mathbb{R}^{n}\right)} \le C\left\|f\right\|_{H\mathcal{M}_{q}^{p}\left(\mathbb{R}^{n}\right)}.
	\end{align*}
\end{theorem}

\begin{theorem}\label{th:Morrey spaces-MB-R Hx-WX}
	\textup{Let $q\in\left(0,1\right]$, $p\in\left(0,\infty\right)$ with $q\le p$, and $\delta>\frac{n-1}{2}$. Let $B_{\ast}^{\delta}$ be the maximal Bochner-Riesz means with $\delta$ order on $\mathbb{R}^{n}$. If $q\in\left[\frac{2n}{n+1+2\delta},1\right]$, then $B_{\ast}^{\delta}$ has a unique extension on $H\mathcal{M}_{q}^{p}\left(\mathbb{R}^{n}\right)$. Moreover, there exists a positive constant $C$ such that for any $f\in H\mathcal{M}_{q}^{p}\left(\mathbb{R}^{n}\right)$,}
	\begin{align*}
		\left\|B_{\ast}^{\delta}\left(f\right)\right\|_{W\mathcal{M}_{q}^{p}\left(\mathbb{R}^{n}\right)} \le C\left\|f\right\|_{H\mathcal{M}_{q}^{p}\left(\mathbb{R}^{n}\right)}.
	\end{align*}
\end{theorem}

\subsection*{Acknowledgments} 
This project is supported by the China Postdoctoral Science Foundation (Grant No. 2023T160296), National Natural Science Foundation of China (Grant No. 11901309), 
and Natural Science Foundation of Nanjing University of Posts and Telecommunications (Grant No. NY222168).

\subsection*{Data Availability} No data available for this publication.

\section*{\bf Declarations}

\subsection*{Conflict of interest} The authors declare that they have no conflict of interest.

\bigskip
\noindent Jian Tan(Corresponding author)\\
\noindent School of Science,
Nanjing University of Posts and Telecommunications,
Nanjing 210023, People's Republic of China\\
Department of Mathematics, Nanjing University,
Nanjing 210093, People's Republic of China   

\noindent {\it E-mail address}: \texttt{tj@njupt.edu.cn}(J. Tan)\\

\noindent Linjing Zhang\\
\noindent School of Science,
Nanjing University of Posts and Telecommunications,
Nanjing 210023, People's Republic of China\\
\noindent {\it E-mail address}: \texttt{shangqingbizhi@163.com}(L. Zhang)


\begin{thebibliography}{99}
\bibitem{AndersenJohn}
\newblock Andersen K. F. and John R. T.,
\newblock \emph{Weighted inequalities for vector-valued maximal functions and singular integrals}.
\newblock Studia Math., 1980, 69: 19-31.\vspace{-0.3cm}
	
\bibitem{BennettSharpley}
\newblock Bennett C. and Sharpley R.,
\newblock \emph{Interpolation of Operators}.
\newblock Pure and Applied Mathematics, vol. 129. Boston: Academic Press, 1988.\vspace{-0.3cm}
	
\bibitem{Bochner}
\newblock Bochner S.,
\newblock \emph{Summation of multiple Fourier series by spherical means}.
\newblock Trans. Amer. Math. Soc., 1936, 40(2): 175-207.\vspace{-0.3cm}	
	
\bibitem{CWYZ}
\newblock Chang, D.-C., Wang, S., Yang, D., Zhang, Y., 
\newblock{\it Littlewood--Paley characterizations of Hardy-type spaces associated with ball quasi-Banach function spaces.} 
\newblock Complex Anal. Oper. Theory, 2020, 14, 1--33. (Paper No. 40)\vspace{-0.3cm}	
	
\bibitem{ChenLau}
\newblock Chen, Y. Z. and Lau, K.-S., 
\newblock \emph{Some new classes of Hardy spaces}.
\newblock J. Funct. Anal. 1989, 84: 255-278.\vspace{-0.3cm}
	
\bibitem{CoifmanRochbergWeiss}
\newblock Coifman R., Rochberg R. and Weiss G.,
\newblock \emph{Fractorization theorems for Hardy spaces in several variables}.
\newblock Ann. of Math., 1976, 103: 611-635.\vspace{-0.3cm}
	
\bibitem{CruzMartell} 
Cruz-Uribe D., Martell J. M. and P\'erez C., 
\newblock \emph{Weights, Extrapolation and the Theory of Rubio de Francia}. 
\newblock Operator Theory: Advances and Applications, vol. 215. Basel: Birkh{\"a}user/Springer Basel AG, 2011.\vspace{-0.3cm}
	
\bibitem{FeffermanStein1971}
\newblock Fefferman C. and Stein E. M.,
\newblock \emph{Some maximal inequalities}.
\newblock Amer. J. Math., 1971, 93: 107-115.\vspace{-0.3cm}  
	   
\bibitem{FeffermanStein1972} 
\newblock Fefferman C. and Stein E. M.,
\newblock \emph{$H^{p}$ spaces of several variables}.
\newblock Acta. Math., 1972, 129: 137-193.\vspace{-0.3cm}
	
\bibitem{Garcia-Cuerva}\newblock Garc{\'i}a-Cuerva J.,\newblock \emph{Hardy spaces and Beurling algebras}.\newblock J. London Math. Soc. 1989, 39(2): 499-513.\vspace{-0.3cm}
	
\bibitem{Garcia-CuervaHerrero}\newblock Garc{\'i}a-Cuerva J. and Herrero M.-J. L.,\newblock \emph{A theory of Hardy spaces associated to Herz spaces}.\newblock Proc. London Math. Soc. 1994, 69(3): 605-628.\vspace{-0.3cm}
	
\bibitem{Grafakos}
\newblock Grafakos L.,
\newblock \emph{Classical Fourier Analysis, 3rd ed}.
\newblock Graduate Texts in Mathematics, vol. 249. New York: Springer, 2014.\vspace{-0.3cm}
	
\bibitem{Ho2016}
\newblock Ho K.-P.,
\newblock \emph{Singular integral operators, John–Nirenberg inequalities and Triebel--Lizorkin type spaces on weighted Lebesgue spaces with variable exponents}. 
\newblock Rev. Un. Mat. Argentina, 2016, 57(1): 85–101.\vspace{-0.3cm}
    
\bibitem{Ho2017}
\newblock Ho K.-P.,
\newblock \emph{Atomic decompositions and Hardy's inequality on weak Hardy--Morrey spaces}.
\newblock Sci. China Math., 2017, 60: 449-468.\vspace{-0.3cm}
	
\bibitem{Ho2019}
\newblock Ho K.-P.,
\newblock \emph{Sublinear operators on weighted Hardy spaces with variable exponents}.
\newblock Forum. Math., 2019, 31: 607-617.\vspace{-0.3cm}
	
\bibitem{Ho2021}
\newblock Ho K.-P.,
\newblock \emph{Singular integral operators and sublinear operators on Hardy local Morrey spaces with variable exponents}.
\newblock Bull. Sci. Math., 2021, 171: 103033.\vspace{-0.3cm}
	
\bibitem{Ho2021-2}
\newblock Ho K.-P.,
\newblock \emph{Operators on Orlicz-slice spaces and Orlicz-slice Hardy spaces}.
\newblock J. Math. Anal. Appl., 2021, 503: 125279.\vspace{-0.3cm}

\bibitem{Hotoappear}
\newblock Ho K.-P.,
\newblock \emph{Calder\'on--Zygmund operators, Bochner--Riesz means and Parametric Marcin- kiewicz integrals on Hardy--Morrey spaces with variable exponents.} 
\newblock Kyoto J. Math, 2023, 63, no.2, 335-351.\vspace{-0.3cm}
	
\bibitem{Izuki}
\newblock Izuki M.,
\newblock \emph{Vector-valued inequalities on Herz spaces and characterizations of Herz--Sobolev spaces with variable exponent}.
\newblock Glas. Mat. Ser. 2010, III 45(65): 475-503.\vspace{-0.3cm}
	
\bibitem{JiaoZuoZhouWu}\newblock Jiao Y., Zuo Y., Zhou D. and Wu L.,\newblock \emph{Variable Hardy--Lorentz spaces $H^{p\left(\cdot\right),q}\left(\mathbb{R}^{n}\right)$}.\newblock Math. Nachr., 2019, 292: 309–349.\vspace{-0.3cm}
	\newblock
	\bibitem{Lee}\newblock Lee M.-Y.,\newblock \emph{Weighted norm inequalities of Bochner--Riesz means}.\newblock J. Math. Anal. Appl., 2006, 324: 1274-1281.\vspace{-0.3cm}
	
\bibitem{LiYangHuang}\newblock Li Y., Yang D. and Huang L.,\newblock \emph{Real-Variable Theory of Hardy Spaces Associated with Generalized Herz Spaces of Rafeiro and Samko}.\newblock Lecture Notes in Mathematics, 2022.\vspace{-0.3cm}
	
\bibitem{LindenstraussTzafriri}\newblock Lindenstrauss J. and Tzafriri L.,\newblock \emph{Classical Banach Spaces II: Function Spaces}.\newblock Ergebnisse der Mathematik und ihrer Grenzgebiete, vol. 97. Berlin-New York: Springer-Verlag, 1979.\vspace{-0.3cm}
	
\bibitem{Liu}\newblock Liu H.,\newblock \emph{The weak $H^{p}$ spaces on homogeneous groups. In: Harmonic Analysis}.\newblock Lecture Notes in Mathematics, vol. 1494. Berlin: Springer, 1991, 113-118.\vspace{-0.3cm}
	
\bibitem{LiuYangYuan}\newblock Liu J., Yang D. and Yuan W. ,\newblock \emph{Littlewood--Paley characterizations of anisotropic Hardy-Lorentz spaces}.\newblock Acta Math. Sci, 2018, 38(1): 1-33.\vspace{-0.3cm}
	
\bibitem{LuYangHerz}\newblock Lu S. and Yang D.,\newblock \emph{The weighted Herz--type Hardy spaces and its applications}.\newblock Sci. China (Ser A), 1995, 38: 662-673.\vspace{-0.3cm}
	
\bibitem{Morrey}\newblock Morrey C. B.,\newblock \emph{On the solutions of quasi-linear elliptic partial differential equations}.\newblock Trans. Amer. Math. Soc., 1938, 43: 126-166.\vspace{-0.3cm}
	
\bibitem{Sato}\newblock Sato S.,\newblock \emph{Weak type estimates for some maximal operators on the weighted Hardy spaces}.\newblock Ark. Mat., 1995, 33: 377-384.\vspace{-0.3cm}
	
\bibitem{SawanoHoYang}\newblock Sawano Y., Ho K.-P., Yang D. and Yang S.,\newblock \emph{Hardy spaces for ball quasi-Banach function spaces}.\newblock Dissertationes Math., 2017, 525: 1-102.\vspace{-0.3cm}	
	
\bibitem{Sj}\newblock \newblock Sj{\"o}lin P.,\newblock \emph{Convolution with oscillating kernels in $H^{p}$ spaces}.\newblock J. London Math. Soc., 1981, 23: 442-454.\vspace{-0.3cm}
	
\bibitem{SteinTaiblesonWeiss}\newblock Stein E. M., Taibleson M. H. and Weiss G.,\newblock \emph{Weak type estimates for maximal operators on certain $H^{p}$ classes}.\newblock Rend. Circ. Mat. Palermo 2, (Suppl. 1), 1981, 81-97.\vspace{-0.3cm}
	
\bibitem{SunYangYuan}\newblock Sun J., Yang D. and Yuan W.,\newblock \emph{Weak Hardy spaces associated with ball quasi-Banach function spaces on spaces of homogeneous type: decompositions, real interpolation, and Calder{\'o}n--Zygmund operators}. \newblock J. Geom. Anal., 2022, 32: 1-85.\vspace{-0.3cm}
	
\bibitem{SYY}\newblock Sun, J., Yang, D., Yuan, W., {\it Molecular characterization of weak Hardy spaces associated with ball quasi-Banach function spaces on spaces of homogeneous type with its application to Littelwood--Paley function characterization}. Forum Math., 2022, 34, 1539--1589.\vspace{-0.3cm}
	
\bibitem{Tan19}
\newblock Tan J., 
\newblock \emph{Atomic decompositions of localized Hardy spaces with variable exponents and applications}. 
\newblock {J.~Geom.~Anal.}, 2019, {29}, no. 1, 799-827.\vspace{-0.3cm}
    
\bibitem{Tan}
\newblock Tan J.,
\newblock \emph{Some Hardy and Carleson measure spaces estimates for Bochner--Riesz means}. 
\newblock {Math. Inequal. Appl.}, 2020, {23}, no. 3, 1027-1039.\vspace{-0.3cm}
	
\bibitem{Tan21}
\newblock Tan J.,
\newblock \emph{Weighted Hardy and Carleson measure spaces estimates for fractional integrations}. 
\newblock {Publ.~Math.~Debrecen.}, 2021, {98}, no. 3-4, 313-330.\vspace{-0.3cm}	
    
\bibitem{ZhangTan}
\newblock Tan J.  and Zhang L.,
\newblock \emph{Bochner--Riesz Means on Hardy Spaces Associated with Ball quasi-Banach function spaces}. 
\newblock  Mediterr. J. Math. 2023, {20}, no.5, Paper No. 240, 1-24.\vspace{-0.3cm}
	
\bibitem{TangXu}\newblock Tang L. and Xu J.,
\newblock \emph{Some properties of Morrey type Besov--Triebel spaces}.
\newblock Math. Nachr., 2005, 278: 904-917.\vspace{-0.3cm}
	
\bibitem{Wang}\newblock Wang H.,\newblock \emph{The boundedness of Bochner--Riesz operator on the weighted weak Hardy space}.\newblock Acta Math. China, 2013, 56(4): 506-518.\vspace{-0.3cm}
	
\bibitem{WangLiuWang}\newblock Wang A., Liu X. and Wang W.,\newblock \emph{Bochner-Riesz Operators on Weak Musielak--Orlicz Hardy Spaces}.\newblock Adv. Math. (China), 2019, 48(5): 565-574.\vspace{-0.3cm}
	
\bibitem{WangQiuWang}\newblock Wang W., Qiu X., Wang A. and Li B.,\newblock \emph{An Estimate for Maximal Bochner--Riesz Means on Musielak--Orlicz Hardy Spaces}.\newblock J. Math. (Wuhan), 2019, 39(5): 694-704.\vspace{-0.3cm}
	
\bibitem{WangYangYang}\newblock Wang F., Yang D. and Yang S.,\newblock \emph{Applications of Hardy spaces associated with ball quasi-Banach function spaces}.\newblock Results Math., 2020, 75: 26, 1-58.\vspace{-0.3cm}
	
\bibitem{WYYZ}
\newblock Wang, S., Yang, D., Yuan, W., Zhang, Y. 
\newblock {\it Weak Hardy-type spaces associated with ball quasi-Banach function spaces II: Littlewood--Paley characterizations and real interpolation.}
\newblock  J. Geom. Anal. 2021, 31, 631--696.\vspace{-0.3cm}
	
\bibitem{YZ}\newblock Yang D. and Zhou Y., 
\newblock {\it A boundedness criterion via atoms for linear operators in Hardy spaces}. 
\newblock {Constr.~Approx.}, 2009, {29}, no. 2, 207-218.\vspace{-0.3cm}

\bibitem{ZhangYangYuan}
\newblock Zhang Y., Yang D., Yuan W. and Wang S.,
\newblock \emph{Weak Hardy-type spaces associated with ball quasi-Banach function spaces I: Decompositions with applications to boundedness of Calder{\'o}n--Zygmund operators}.
\newblock Sci. China Math., 2021, 64: 2007-2064.\vspace{-0.3cm}
	
\end{thebibliography}
\end{document}